\documentclass{article}

\usepackage[accepted]{icml2021}

\usepackage{microtype}
\usepackage{subfigure}
\usepackage[linesnumbered,boxed,titlenumbered,ruled,nofillcomment]{algorithm2e}
\usepackage[mathscr]{eucal}
\usepackage{bm}
\usepackage{amsthm}
\usepackage{amssymb}
\usepackage{mathtools} \mathtoolsset{showonlyrefs=true} \MakeRobust{\eqref}
\usepackage{graphicx}
\usepackage{xr-hyper}
\usepackage{hyperref}
\usepackage[dvipsnames]{xcolor}
\usepackage[utf8]{inputenc}
\usepackage{booktabs}
\usepackage{siunitx} 
\usepackage{pifont} \newcommand{\xmark}{\ding{55}} 

\SetInd{0.5em}{0.5em} 

\hypersetup{
	colorlinks=true,
	citecolor=MidnightBlue,
	urlcolor=MidnightBlue,
	linkcolor=MidnightBlue,
}

\SetCommentSty{mycommfont}

\theoremstyle{plain}
\newtheorem{theorem}{Theorem} 

\newtheorem{lemma}[theorem]{Lemma}

\theoremstyle{definition}
\newtheorem{remark}[theorem]{Remark}

\newtheorem{definition}[theorem]{Definition}

\newcommand{\Ge}{\boldsymbol{\mathscr{G}}}

\newcommand{\Xe}{\boldsymbol{\mathscr{X}}}

\newcommand{\Ze}{\boldsymbol{\mathscr{Z}}}

\newcommand{\F}{\textup{F}}

\newcommand{\zerobf}[0]{\bm{0}}

\newcommand{\Omegabf}[0]{\bm{\Omega}}

\newcommand{\Sigmabf}[0]{\bm{\Sigma}}
\newcommand{\Abf}[0]{\bm{A}}
\newcommand{\Bbf}[0]{\bm{B}}
\newcommand{\Cbf}[0]{\bm{C}}

\newcommand{\Gbf}[0]{\bm{G}}

\newcommand{\Ibf}[0]{\bm{I}}

\newcommand{\Mbf}[0]{\bm{M}}

\newcommand{\Qbf}[0]{\bm{Q}}
\newcommand{\Rbf}[0]{\bm{R}}
\newcommand{\Sbf}[0]{\bm{S}}

\newcommand{\Ubf}[0]{\bm{U}}
\newcommand{\Vbf}[0]{\bm{V}}
\newcommand{\Wbf}[0]{\bm{W}}
\newcommand{\Xbf}[0]{\bm{X}}
\newcommand{\Ybf}[0]{\bm{Y}}
\newcommand{\Zbf}[0]{\bm{Z}}

\newcommand{\ebf}[0]{\bm{e}}

\newcommand{\pbf}[0]{\bm{p}}
\newcommand{\qbf}[0]{\bm{q}}

\newcommand{\vbf}[0]{\bm{v}}

\newcommand{\xbf}[0]{\bm{x}}
\newcommand{\ybf}[0]{\bm{y}}

\newcommand\Eb{\mathbb{E}}

\newcommand\Pb{\mathbb{P}}

\newcommand\Rb{\mathbb{R}}

\newcommand{\Dc}[0]{\mathcal{D}}

\DeclareMathOperator*{\argmin}{arg\,min}
\DeclareMathOperator*{\TR}{\operatorname{TR}}
\newcommand{\defeq}{\stackrel{\text{\tiny \textnormal{def}}}{=}}  

\newcommand{\rank}[0]{\operatorname{rank}}
\newcommand{\noiter}{\textup{\#iter}}

\newcommand*{\OPT}{\textup{OPT}}
\newcommand{\range}[0]{\operatorname{range}}

\icmltitlerunning{A Sampling-Based Method for Tensor Ring Decomposition}

\begin{document}

\twocolumn[
\icmltitle{A Sampling-Based Method for Tensor Ring Decomposition}

\icmlsetsymbol{equal}{*}

\begin{icmlauthorlist}
\icmlauthor{Osman Asif Malik}{cu}
\icmlauthor{Stephen Becker}{cu}
\end{icmlauthorlist}

\icmlaffiliation{cu}{Department of Applied Mathematics, University of Colorado Boulder, Boulder, CO, USA}

\icmlcorrespondingauthor{Osman Asif Malik}{osman.malik@colorado.edu}

\icmlkeywords{tensor decomposition, tensor ring, matrix sketching, leverage score sampling, feature extraction}

\vskip 0.3in
]

\printAffiliationsAndNotice{} 

\begin{abstract}
We propose a sampling-based method for computing the tensor ring (TR) decomposition of a data tensor.
The method uses leverage score sampled alternating least squares to fit the TR cores in an iterative fashion.
By taking advantage of the special structure of TR tensors, we can efficiently estimate the leverage scores and attain a method which has complexity sublinear in the number of input tensor entries.
We provide high-probability relative-error guarantees for the sampled least squares problems.
We compare our proposal to existing methods in experiments on both synthetic and real data.
Our method achieves substantial speedup---sometimes two or three orders of magnitude---over competing methods, while maintaining good accuracy.
We also provide an example of how our method can be used for rapid feature extraction.
\end{abstract}

\section{Introduction} \label{sec:intro}

Tensor decomposition has recently found many applications in machine learning and related areas.
Examples include parameter reduction in neural networks \citep{novikov2015, garipov2016, yang2017, yu2017, ye2018}, 
understanding the expressivity of deep neural networks \citep{cohen2016, khrulkov2018expressive},
supervised learning \citep{stoudenmire2017a, novikov2017}, 
filter learning \citep{hazan2005, rigamonti2013},
image factor analysis and recognition \citep{vasilescu2002, liu2019},
scaling up Gaussian processes \citep{izmailov2018},
multimodal feature fusion \citep{hou2019}, 
natural language processing \citep{lei2014}, 
data mining \citep{papalexakis2016},
feature extraction \citep{bengua2015}, 
and tensor completion \citep{wang2017}.
The CP and Tucker decompositions are the two most popular tensor decompositions that each have strengths and weaknesses \citep{kolda2009}. 
The number of parameters in the CP decomposition scales linearly with the tensor dimension, but finding a good decomposition numerically can be challenging.
The Tucker decomposition is usually easier to work with numerically, but the number of parameters grows exponentially with tensor dimension.
To address these challenges, decompositions based on tensor networks have recently gained in popularity.
The tensor train (TT) is one such decomposition \citep{oseledets2011}.
The number of parameters in a TT scales linearly with the tensor dimension.
The tensor ring (TR) decomposition is a generalization of the TT decomposition.
It is more expressive than the TT decomposition (see Remark~\ref{remark:expressiveness}), but maintains the linear scaling of the number of parameters with increased tensor dimension.
Although the TR decomposition is known to have certain numerical stability issues it usually works well in practice, achieving a better compression ratio and lower storage cost than the TT decomposition \citep{mickelin2020}.

It is crucial to have efficient algorithms for computing the decomposition of a tensor.
This task is especially challenging for large and high-dimensional tensors.
In this paper, we present a sampling-based method for TR decomposition which has complexity sublinear in the number of input tensor entries.
The method fits the TR core tensors by iteratively updating them in an alternating fashion.
Each update requires solving a least squares problem.
We use leverage score sampling to speed up the solution of these problems and avoid having to consider all elements in the input tensor for each update.
The system matrices in the least squares problems have a particular structure that we take advantage of to efficiently estimate the leverage scores.
To summarize, we make the following contributions:
\begin{itemize}
	\item Propose a sampling-based method for TR decomposition which has complexity sublinear in the number of input tensor entries.
	\item Provide relative-error guarantees for the sampled least squares problems we use in our algorithm.
	\item Compare our proposed algorithm to four other methods in experiments on both synthetic and real data.
	\item Provide an example of how our method can be used for rapid feature extraction.
\end{itemize}

\section{Related Work} \label{sec:related-work}

Leverage score sampling has been used previously for tensor decomposition.
\citet{cheng2016} use it to reduce the size of the least squares problems that come up in the alternating least squares (ALS) algorithm for CP decomposition.
They efficiently estimate the leverage scores by exploiting the Khatri--Rao structure of the design matrices in these least squares problems and prove additive-error guarantees.
\citet{larsen2020} further improve this method by combining repeatedly sampled rows and by using a hybrid deterministic--random sampling scheme.
These techniques make the method faster without sacrificing accuracy.
Our paper differs from these previous works since we consider a different kind of tensor decomposition.
Since the TR decomposition has core tensors instead of factor matrices, the leverage score estimation formulas used in the previous works cannot be used for TR decomposition.

There are previous works that develop randomized algorithms for TR decomposition.
\citet{yuan2019} develop a method which first applies a randomized variant of the HOSVD Tucker decomposition, followed by a TR decomposition of the core tensor using either a standard SVD- or ALS-based algorithm.
Finally, the TR cores are contracted appropriately with the Tucker factor matrices to get a TR decomposition of the original tensor.
\citet{ahmadi-asl2020} present several TR decomposition algorithms that incorporate randomized SVD. 
In Section~\ref{sec:experiments}, we compare our proposed algorithm to methods by \citet{yuan2019} and \citet{ahmadi-asl2020} in experiments.

Papers that develop randomized algorithms for other tensor decompositions include the works by \citet{wang2015}, \citet{battaglino2018}, \citet{yang2018} and \citet{aggour2020} for the CP decomposition; \citet{biagioni2015} and \citet{malik2020a} for tensor interpolative decomposition; \citet{drineas2007}, \citet{tsourakakis2010}, \citet{dacosta2016}, \citet{malik2018}, \citet{sun2020} and \citet{minster2020} for the Tucker decomposition; \citet{zhang2018} and \citet{tarzanagh2018} for t-product-based decompositions; and \citet{huber2017} and \citet{che2019} for the TT decomposition.

Skeleton approximation and other sampling-based methods for TR decomposition include the works by \citet{espig2012} and \citet{khoo2019}.
\citet{espig2012} propose a skeleton/cross approximation type method for the TR format with complexity linear in the dimensionality of the tensor.
\citet{khoo2019} propose an ALS-based scheme for constructing a TR decomposition based on only a few samples of the input tensor.
Papers that use skeleton approximation and sampling techniques for other types of tensor decomposition include those by \citet{mahoney2008}, \citet{oseledets2008}, \citet{oseledets2010a}, \citet{caiafa2010}, and \citet{friedland2011}.

\section{Preliminaries} \label{sec:preliminaries}

By tensor, we mean a multidimensional array of real numbers.
Boldface Euler script letters, e.g.\ $\Xe$, denote tensors of dimension 3 or greater; bold uppercase letters, e.g.\ $\Xbf$, denote matrices; bold lowercase letters, e.g.\ $\xbf$, denote vectors; and regular lowercase letters, e.g.\ $x$, denote scalars. 
Elements of tensors, matrices and vectors are denoted in parentheses. 
For example, $\Xe(i_1, i_2, i_3)$ is the element on position $(i_1, i_2, i_3)$ in the 3-way tensor $\Xe$, and $\ybf(i_1)$ is the element on position $i_1$ in the vector $\ybf$.
A colon is used to denote all elements along a dimension.
For example, $\Xbf(i,:)$ denotes the $i$th row of the matrix $\Xbf$.
For a positive integer $I$, define $[I] \defeq \{ 1, \ldots, I \}$.
For indices $i_1 \in [I_1], \ldots, i_N \in [I_N]$, the notation $\overline{i_1 i_2 \cdots i_N} \defeq \sum_{n=1}^N i_n \prod_{j=1}^{n-1} I_j$ will be helpful when working with reshaped tensors.
$\|\cdot\|_\F$ denotes the Frobenius norm of matrices and tensors, and $\|\cdot\|_2$ denotes the Euclidean norm of vectors.

\subsection{Tensor Ring Decomposition}

Let $\Xe \in \Rb^{I_1 \times \cdots \times I_N}$ be an $N$-way tensor. For $n \in [N]$, let $\Ge^{(n)} \in \Rb^{R_{n-1} \times I_n \times R_n}$ be 3-way tensors with $R_0 = R_N$.
A rank-$(R_1, \ldots, R_N)$ TR decomposition of $\Xe$ takes the form
\begin{equation} \label{eq:tr-decomposition-elementwise}
	\Xe(i_1, \ldots, i_N) = \sum_{r_1, \ldots, r_N} \prod_{n=1}^N \Ge^{(n)}(r_{n-1}, i_n, r_n),
\end{equation}
where each $r_n$ in the sum goes from $1$ to $R_n$ and $r_0 = r_{N}$. 
$\Ge^{(1)}, \ldots, \Ge^{(N)}$ are called \emph{core tensors}. 
We will use $\TR((\Ge^{(n)})_{n=1}^N)$ to denote a tensor ring with cores $\Ge^{(1)}, \ldots, \Ge^{(N)}$.
The name of the decomposition comes from the fact that it looks like a ring in tensor network\footnote{See \citet{cichocki2016, cichocki2017} for an introduction to tensor networks.} notation; see Figure~\ref{fig:trd}. 
The TT decomposition is a special case of the TR decomposition with $R_0 = R_N = 1$, which corresponds to severing the connection between $\Ge^{(1)}$ and $\Ge^{(5)}$ in Figure~\ref{fig:trd}.
The TR decomposition can therefore also be interpreted as a sum of tensor train decompositions with the cores $\Ge^{(2)},\ldots,\Ge^{(N-1)}$ in common.

\begin{remark}[Expressiveness of TR vs TT] \label{remark:expressiveness}
	Since the TT decomposition is a special case of the TR decomposition, any TT is also a TR with the same number of parameters.
	So the TR decomposition is \emph{at least as expressive} as the TT decomposition. 
	To see that the TR decomposition is \emph{strictly more expressive} than the TT decomposition, consider the following example: 
	Let $\Xe = \TR(\Ge, \Ge, \Ge, \Ge)$ where $\Ge \in \Rb^{2 \times 2 \times 2}$ has entries $1,2,\ldots,8$ ($\Ge = \verb|reshape(1:8,2,2,2)|$ in Matlab).
	This TR of $\Xe$ requires $4\cdot(2\cdot2\cdot2) = 32$ parameters. 
	The TT ranks of $\Xe$ are $(2,4,2)$ (computed via Theorem 8.8 by \citet{ye2018b}).
	An exact TT decomposition of $\Xe$ therefore requires $2\cdot(2\cdot2) + 2\cdot(2\cdot2\cdot4) = 40$ parameters. 
\end{remark}

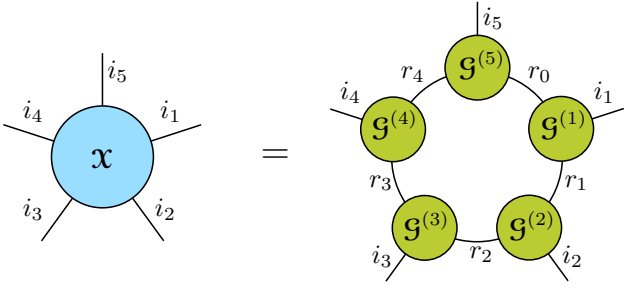
\begin{figure}[h]
	\centering
	\includegraphics[width=1\columnwidth]{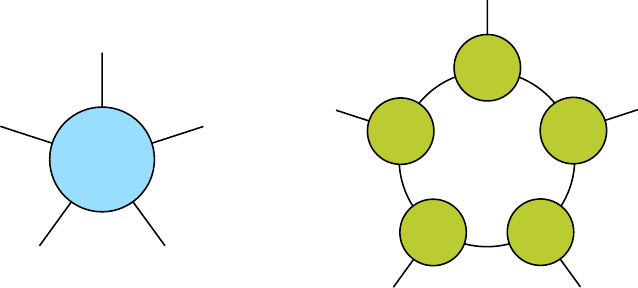}
	\caption{
		Tensor ring decomposition of $5$-way tensor. 
	}
	\label{fig:trd}
	\begin{picture}(0,0)
		\put(-84,76){$\Xe$}
		\put(-58,95){$i_1$}
		\put(-60,58){$i_2$}
		\put(-110,58){$i_3$}
		\put(-110,95){$i_4$}
		\put(-78,110){$i_5$}
		
		\put(-20,78){\Large$=$}
		
		\put(86,88){$\Ge^{(1)}$}
		\put(74,50){$\Ge^{(2)}$}
		\put(34,50){$\Ge^{(3)}$}
		\put(22,88){$\Ge^{(4)}$}
		\put(55,111){$\Ge^{(5)}$}
		\put(106,102){$i_1$}
		\put(94,40){$i_2$}
		\put(22,40){$i_3$}
		\put(10,102){$i_4$}
		\put(64,130){$i_5$}
		\put(95,69){$r_1$}
		\put(59,42){$r_2$}
		\put(21,69){$r_3$}
		\put(33,110){$r_4$}
		\put(81,110){$r_0$}
	\end{picture}
\end{figure}

The problem of fitting $\TR((\Ge^{(n)})_{n=1}^N)$ to a data tensor $\Xe$ can be written as the minimization problem
\begin{equation} \label{eq:TRD-minimization-problem}
	\argmin_{\Ge^{(1)}, \ldots, \Ge^{(N)}} \| \TR((\Ge^{(n)})_{n=1}^N) - \Xe \|_\F,
\end{equation}
where the size of each core tensor is fixed.
There are two common approaches to this fitting problem: one is SVD based and the other uses ALS. 
We describe the latter below since it is what we use in our work, and refer to \citet{zhao2016} and \citet{mickelin2020} for a description of the SVD-based approach.

\begin{definition} \label{def:unfolding}
	The \emph{classical mode-$n$ unfolding} of $\Xe$ is the matrix $\Xbf_{(n)} \in \Rb^{I_n \times \prod_{j \neq n} I_j}$ defined elementwise via
	\begin{equation}
		\Xbf_{(n)} (i_n, \overline{i_1 \cdots i_{n-1} i_{n+1} \cdots i_N}) \defeq \Xe(i_1, \ldots, i_N).
	\end{equation}
	The \emph{mode-$n$ unfolding} of $\Xe$ is the matrix $\Xbf_{[n]} \in \Rb^{I_n \times \prod_{j \neq n} I_j}$ defined elementwise via 
	\begin{equation}
		\Xbf_{[n]} (i_n, \overline{i_{n+1} \cdots i_N i_1 \cdots i_{n-1}} ) \defeq \Xe(i_1, \ldots, i_N).
	\end{equation}
\end{definition}
\begin{definition} \label{def:subchain}
	By merging all cores except the $n$th, we get a \emph{subchain tensor} $\Ge^{\neq n} \in \Rb^{R_n \times \prod_{j \neq n} I_j \times R_{n-1}}$ defined elementwise via
	\begin{equation}
	\begin{aligned}
		&\Ge^{\neq n}(r_n, \overline{i_{n+1} \ldots i_N i_1 \ldots i_{n-1}}, r_{n-1}) \\
		&\defeq \sum_{\substack{r_1, \ldots, r_{n-2}\\r_{n+1}, \ldots, r_N}} \prod_{\substack{j = 1\\j \neq n}}^N \Ge^{(j)}(r_{j-1}, i_j, r_j).
	\end{aligned}
	\end{equation}
\end{definition}
It follows directly from Theorem~3.5 in \citet{zhao2016} that the objective in \eqref{eq:TRD-minimization-problem} can be rewritten as
\begin{equation}
	\| \TR((\Ge^{(n)})_{n=1}^N) - \Xe \|_\F = \| \Gbf_{[2]}^{\neq n} \Gbf_{(2)}^{(n)\top} - \Xbf_{[n]}^\top \|_\F.
\end{equation}
The ALS approach to finding an approximate solution to \eqref{eq:TRD-minimization-problem} is to keep all cores except the $n$th fixed and solve the least squares problem above with respect to that core, and then repeat this multiple times for each $n \in [N]$ until some termination criterion is met.
We summarize the approach in Algorithm~\ref{alg:TR-ALS}.
The initialization on line~\ref{line:tr-als:initialize-cores} can be done by, e.g., drawing each core entry independently from a standard normal distribution.
Notice that the first core does not need to be initialized since it will be immediately updated.
Normalization steps can also be added to the algorithm, but such details are omitted here.
Throughout this paper we make the reasonable assumption that $\Ge^{\neq n} \neq \zerobf$ for all $n \in [N]$ during the execution of the ALS-based methods. 
This is discussed further in Remark~\ref{SUPP:remark:nonzero-cores} in the supplement.

\begin{algorithm}
	\caption{TR-ALS \citep{zhao2016}}
	\label{alg:TR-ALS}
	\DontPrintSemicolon
	\SetAlgoNoLine
	\KwIn{$\Xe \in \Rb^{I_1 \times \cdots \times I_N}$, target ranks $(R_1, \ldots, R_N)$}
	\KwOut{TR cores $\Ge^{(1)}, \ldots, \Ge^{(N)}$}
	Initialize cores $\Ge^{(2)}, \ldots, \Ge^{(N)}$ \label{line:tr-als:initialize-cores}\;
	\Repeat{termination criteria met}{
		\For{$n = 1,\ldots, N$}{
			Compute $\Gbf_{[2]}^{\neq n}$ from cores \label{line:tr-als:compute-subchain}\;
			Update $\Ge^{(n)} = \argmin_{\Ze} \| \Gbf_{[2]}^{\neq n} \Zbf_{(2)}^\top - \Xbf_{[n]}^\top \|_\F$ \label{line:tr-als:ls}\;
		}	
	}
	\Return{$\Ge^{(1)}, \ldots, \Ge^{(N)}$}\;
\end{algorithm}

\subsection{Leverage Score Sampling}

Let $\Abf \in \Rb^{K \times L}$ and $\ybf \in \Rb^K$ where $K \gg L$.
Solving the overdetermined least squares problem $\min_{\xbf} \|\Abf \xbf - \ybf\|_2$ using standard methods costs $O(KL^2)$.
One approach to reducing this cost is to randomly sample and rescale $J \ll K$ of the rows of $\Abf$ and $\ybf$ according to a carefully chosen probability distribution and then solve this smaller system instead.
This will reduce the solution cost to $O(JL^2)$.
\begin{definition} \label{def:leverage-score}
	For a matrix $\Abf \in \Rb^{K \times L}$, let $\Ubf \in \Rb^{K \times \rank(\Abf)}$ contain the left singular vectors of $\Abf$. The $i$th \emph{leverage score} of $\Abf$ is defined as $\ell_i(\Abf) \defeq \|\Ubf(i,:)\|_2^2$ for $i \in [K]$.
\end{definition}
\begin{definition} \label{def:leverage-score-sampling-matrix}
	Let $\qbf \in \Rb^K$ be a probability distribution on $[K]$, and let $\vbf \in [K]^J$ be a random vector with independent elements satisfying $\Pb(\vbf(j) = i) = \qbf(i)$ for all $(i,j) \in [K] \times [J]$.
	Let $\Omegabf \in \Rb^{J \times K}$ and $\Rbf \in \Rb^{J \times J}$ be a random sampling matrix and a diagonal rescaling matrix, respectively, defined as $\Omegabf(j, :) \defeq \ebf_{\vbf(j)}^\top$ and $\Rbf(j,j) \defeq 1/\sqrt{J \qbf(\vbf(j))}$ for each $j \in [J]$, where $\ebf_i$ is the $i$th column of the $K \times K$ identity matrix.
	We say that $\Sbf \defeq \Rbf \Omegabf \in \Rb^{J \times K}$ is a \emph{sampling matrix with parameters $(J, \qbf)$}, or $\Sbf \sim \Dc(J, \qbf)$ for short.
	Let $\Abf \in \Rb^{K \times L}$ be a nonzero matrix, let $\pbf \in \Rb^K$ be a probability distribution on $[K]$ with entries $\pbf(i) \defeq \ell_i(\Abf)/\rank(\Abf)$, and suppose $\beta \in (0,1]$.
	We say that $\Sbf \sim \Dc(J, \qbf)$ is a \emph{leverage score sampling matrix for $(\Abf, \beta)$} if $\qbf(i) \geq \beta \pbf(i)$ for all $i \in [K]$.
\end{definition}

Notice that for any $\Abf$ we have $\sum_i \ell_i(\Abf) = \rank(\Abf)$ and therefore $\sum_i \pbf(i) = 1$ as desired for a probability distribution. 
One can show that if $\Sbf \sim \Dc(J, \qbf)$ is a leverage score sampling matrix for $(\Abf, \beta)$ with $J$ sufficiently large, then $\|\Abf \tilde{\xbf} - \ybf\|_2 \leq (1+\varepsilon) \min_{\xbf} \|\Abf \xbf - \ybf\|_2$ with high probability, where $\tilde{\xbf} \defeq \argmin_{\xbf} \|\Sbf \Abf \xbf - \Sbf \ybf\|_2$.
How large to make $J$ depends on the probability of success and the parameters $\varepsilon$ and $\beta$.
For further details on leverage score sampling, see \citet{mahoney2011} and \citet{woodruff2014}.

\section{Tensor Ring Decomposition via Sampling} \label{sec:method}

We propose using leverage score sampling to reduce the size of the least squares problem on line~\ref{line:tr-als:ls} in Algorithm~\ref{alg:TR-ALS}.
A challenge with using leverage score sampling is computing a sampling distribution $\qbf$ which satisfies the conditions in Definition~\ref{def:leverage-score-sampling-matrix}.
One option is to use $\qbf = \pbf$ and $\beta = 1$, with $\pbf$ defined as in Definition~\ref{def:leverage-score-sampling-matrix}.
However, computing $\pbf$ requires finding the left singular vectors of the design matrix which costs the same as solving the original least squares problem.

When the design matrix is $\Gbf_{[2]}^{\neq n}$, a sampling distribution $\qbf$ can be computed much more efficiently directly from the cores $\Ge^{(1)}, \ldots, \Ge^{(n-1)}, \Ge^{(n+1)}, \ldots, \Ge^{(N)}$ without explicitly forming $\Gbf_{[2]}^{\neq n}$.
For each $n \in [N]$, let $\pbf^{(n)} \in \Rb^{I_n}$ be a probability distribution on $[I_n]$ defined elementwise via\footnote{$\rank(\Gbf_{(2)}^{(n)}) \geq 1$ due to our assumption that $\Ge^{\neq n} \neq \zerobf$ for all $n \in [N]$, so \eqref{eq:p-definition} is well-defined; see Remark~\ref{SUPP:remark:nonzero-cores} in the supplement.}
\begin{equation} \label{eq:p-definition}
	\pbf^{(n)} (i_n) \defeq \frac{\ell_{i_n} (\Gbf_{(2)}^{(n)})}{\rank(\Gbf_{(2)}^{(n)})}.
\end{equation} 
Furthermore, let $\qbf^{\neq n} \in \Rb^{\prod_{j \neq n} I_j}$ be a vector defined elementwise via
\begin{equation} \label{eq:q-definition}
	\qbf^{\neq n}(\overline{i_{n+1} \cdots i_N i_1 \cdots i_{n-1}}) \defeq \prod_{\substack{j = 1\\j \neq n}}^{N} \pbf^{(j)}(i_j).
\end{equation}
\begin{lemma} \label{lemma:leverage-score-sampling}
	Let $\beta_n$ be defined as
	\begin{equation}
		\beta_n \defeq \Big( R_{n-1} R_n \prod_{\substack{j=1\\j \notin \{n-1,n\}}}^N R_j^2 \Big)^{-1}.
	\end{equation}
	For each $n \in [N]$, the vector $\qbf^{\neq n}$ is a probability distribution on $[\prod_{j \neq n} I_j]$, and $\Sbf \sim \Dc(J, \qbf^{\neq n})$ is a leverage score sampling matrix for $(\Gbf_{[2]}^{\neq n}, \beta_n)$.
\end{lemma}

This lemma can now be used to prove the following guarantees for a sampled variant of the least squares problem on line~\ref{line:tr-als:ls} of Algorithm~\ref{alg:TR-ALS}.
\begin{theorem} \label{thm:main}
	Let $\Sbf \sim \Dc(J, \qbf^{\neq n})$, $\varepsilon \in (0,1)$, $\delta \in (0,1)$ and $\tilde{\Ze} \defeq \argmin_{\Ze} \| \Sbf \Gbf_{[2]}^{\neq n} \Zbf_{(2)}^\top - \Sbf \Xbf_{[n]}^\top \|_\F$.
	If 
	\begin{equation} \label{eq:J-bound-2}
		J > \Big( \prod_{j=1}^N R_j^2 \Big) \max \Big( \frac{16}{3(\sqrt{2} - 1)^2} \ln\Big(\frac{4 R_{n-1} R_n}{\delta}\Big), \frac{4}{\varepsilon\delta} \Big),
	\end{equation}
	then the following holds with probability at least $1-\delta$:
	\begin{equation}
		\| \Gbf_{[2]}^{\neq n} \tilde{\Zbf}_{(2)}^\top - \Xbf_{[n]}^\top \|_\F \leq (1+\varepsilon) \min_{\Ze} \| \Gbf_{[2]}^{\neq n} \Zbf_{(2)}^\top - \Xbf_{[n]}^\top \|_\F.
	\end{equation}
\end{theorem}

Proofs of Lemma~\ref{lemma:leverage-score-sampling} and Theorem~\ref{thm:main} are given in Section~\ref{SUPP:sec:proofs} of the supplement.
In Algorithm~\ref{alg:TR-ALS-Sampled} we present our proposed TR decomposition algorithm.

\begin{algorithm}
	\caption{{TR-ALS-Sampled} (proposal)}
	\label{alg:TR-ALS-Sampled}
	\DontPrintSemicolon 
	\SetAlgoNoLine
	\KwIn{$\Xe \in \Rb^{I_1 \times \cdots \times I_N}$, target ranks $(R_1, \ldots, R_N)$}
	\KwOut{TR cores $\Ge^{(1)}, \ldots, \Ge^{(N)}$}
	Initialize cores $\Ge^{(2)}, \ldots, \Ge^{(N)}$ \label{line:tr-als-sampled:initialize-cores}\;
	Compute distributions $\pbf^{(2)}, \ldots, \pbf^{(N)}$ via \eqref{eq:p-definition} \label{line:tr-als-sampled:compute-p}\;
	\Repeat{termination criteria met}{
		\For{$n = 1,\ldots,N$}{
			Set sample size $J$ \label{line:tr-als-sampled:J}\;
			Draw sampling matrix $\Sbf \sim \Dc(J, \qbf^{\neq n})$ \label{line:tr-als-sampled:draw-sketch}\;
			Compute $\tilde{\Gbf}_{[2]}^{\neq n} = \Sbf \Gbf_{[2]}^{\neq n}$ from cores \label{line:tr-als-sampled:sample-G}\;
			Compute $\tilde{\Xbf}_{[n]}^\top = \Sbf \Xbf_{[n]}^\top$ \label{line:tr-als-sampled:sample-X}\;
			Update $\Ge^{(n)} = \argmin_{\Ze} \| \tilde{\Gbf}_{[2]}^{\neq n} \Zbf_{(2)}^\top - \tilde{\Xbf}_{[n]}^\top \|_\F$ \label{line:tr-als-sampled:ls}\;
			Update $n$th distribution $\pbf^{(n)}$ via \eqref{eq:p-definition} \label{line:tr-als-sampled:update-p}\;
		}	
	}
	\Return{$\Ge^{(1)}, \ldots, \Ge^{(N)}$}\;
\end{algorithm}

We draw each core entry independently from a standard normal distribution during the initialization on line~\ref{line:tr-als-sampled:initialize-cores} in our experiments.
The sample size on line~\ref{line:tr-als-sampled:J} can be provided as a user input or be determined adaptively; this is discussed further in Section~\ref{sec:adaptive-sample-size}.
The sampling matrix $\Sbf$ is never explicitly constructed. 
Instead, on line~\ref{line:tr-als-sampled:draw-sketch}, we draw and store a realization of the vector $\vbf \in [\prod_{j \neq n} I_j]^J$ in Definition~\ref{def:leverage-score-sampling-matrix}.
Entries of $\vbf$  can be drawn efficiently using Algorithm~\ref{alg:sampling-q}.

\begin{algorithm}
	\caption{Sampling from $\qbf^{\neq n}$ (proposal)}
	\label{alg:sampling-q}
	\DontPrintSemicolon 
	\SetAlgoNoLine
	\KwIn{Distributions $\{\pbf^{(k)}\}_{k \neq n}$}
	\KwOut{Sample $i \sim \qbf^{\neq n}$}
	\lFor{$j \in [N] \setminus n$}{
		draw  realization $i_j \sim \pbf^{(j)}$
	}
	\Return{$i = \overline{i_{n+1} \cdots i_N i_1 \cdots i_{n-1}}$}\;
\end{algorithm}

The sampling on line~\ref{line:tr-als-sampled:sample-G} in Algorithm~\ref{alg:TR-ALS-Sampled} can be done efficiently directly from the cores without forming $\Gbf_{[2]}^{\neq n}$. 
To see this, note that the matrix row $\Gbf_{[2]}^{\neq n}(i, :)$ is the vectorization of the tensor slice $\Ge^{\neq n}(:, i, :)$ due to Definition~\ref{def:unfolding}.
From Definition~\ref{def:subchain}, this tensor slice in turn is given by the sequence of matrix multiplications
\begin{equation}
\begin{aligned}
\Ge^{\neq n}(:, i, :) &= \Ge^{(n+1)}(:, i_{n+1}, :) \cdots \Ge^{(N)}(:, i_{N}, :)\\
&\cdot \Ge^{(1)}(:, i_1, :) \cdots \Ge^{(n-1)}(:, i_{n-1}, :),
\end{aligned}
\end{equation}
where $i = \overline{i_{n+1} \cdots i_N i_1 \cdots i_{n-1}}$ and each ${\Ge^{(j)}(:, i_{j}, :)}$ is treated as an $R_{j-1} \times R_{j}$ matrix.
For a realization of $\vbf \in [\prod_{j \neq n} I_j]^J$, it is therefore sufficient to extract $J$ lateral slices from each of the cores $\Ge^{(1)}, \ldots, \Ge^{(n-1)}, \Ge^{(n+1)}, \ldots, \Ge^{(N)}$ and then multiply together the $j$th slice from each core, for each $j \in [J]$.
This is illustrated in Figure~\ref{fig:core-sampling}.
An algorithm for computing $\tilde{\Gbf}^{\neq n}_{[2]}$ in this way is given in Algorithm~\ref{alg:compute-sketch-G}.

\begin{figure}[h]
	\centering  
	\includegraphics[width=1\columnwidth]{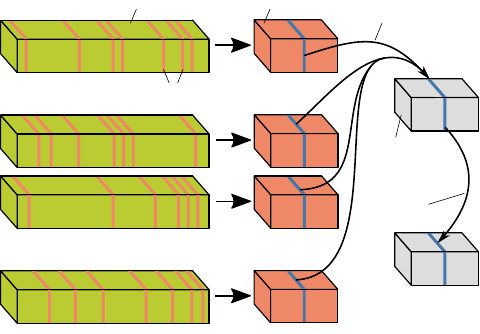}
	\caption{Illustration of how to efficiently construct $\tilde{\Ge}^{\neq n}$ by sampling the core tensors.}
	\label{fig:core-sampling}
	\begin{picture}(0,0)
		\put(-114,200){\footnotesize $\Ge^{(n+1)}$}
		\put(-114,154){\footnotesize $\Ge^{(N)}$}
		\put(-114,86){\footnotesize $\Ge^{(1)}$}
		\put(-114,39){\footnotesize $\Ge^{(n-1)}$}
		
		\put(-65,157){$\vdots$}
		\put(-65,81){$\vdots$}
		\put(24,157){$\vdots$}
		\put(24,81){$\vdots$}
		
		\put(-60,204){{\tiny Full cores}}
		\put(-42,162){{\tiny Samples}}
		
		\put(60,203){{\tiny Multiplication}}
		\put(60,197){{\tiny of lateral slices}}
		
		\put(2,204){{\tiny Sampled cores}}
		
		\put(64,135){{\tiny Sampled}}
		\put(64,129){{\tiny subchain}}
		\put(64,123){{\tiny tensor}}
		
		\put(64,108){{\tiny Rescale}}
		\put(64,101){{\tiny frontal slices}}
				
		\put(86,56){\footnotesize  $\tilde{\Ge}^{\neq n}$}
	\end{picture}
\end{figure}
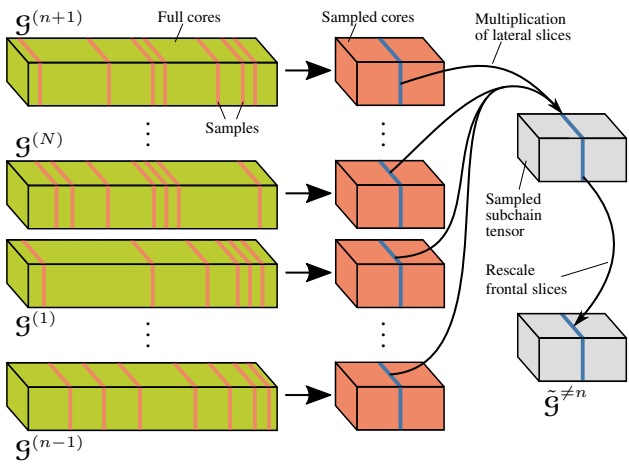 

\begin{algorithm}
	\caption{Computation of $\tilde{\Gbf}_{[2]}^{\neq n}$ (proposal)}
	\label{alg:compute-sketch-G}
	\DontPrintSemicolon 
	\SetAlgoNoLine
	\KwIn{Samples $\vbf$, distributions $\{\pbf^{(k)}\}_{k \neq n}$}
	\KwOut{Sketch $\tilde{\Gbf}_{[2]}^{\neq n} = \Sbf \Gbf_{[2]}^{\neq n}$}
	Initialize $\tilde{\Ge}^{\neq n} = \zerobf \in \Rb^{R_n \times J \times R_{n-1}}$\;
	\For{$j \in [J]$}{
		$i = \vbf(j)$ \tcp*{Where $i = \overline{i_{n+1} \cdots i_N i_1 \cdots i_{n-1}}$} 
		$c = \sqrt{J \prod_{k \neq n} \pbf(i_k)}$\;
		$\tilde{\Ge}^{\neq n}(:,j,:) = \Ge^{(n+1)}(:,i_{n+1},:) \cdots \Ge^{(N)}(:,i_N,:)$\;
		$\;\;\;\;\;\;\;\;\;\;\;\;\;\;\;\;\;\;\;\;\cdot \Ge^{(1)}(:, i_1, :) \cdots \Ge^{(n-1)}(:, i_{n-1}, :) / c$\;
	}
	\Return{$\tilde{\Gbf}_{[2]}^{\neq n}$}\;
\end{algorithm}

The sampling of the data tensor on line~\ref{line:tr-als-sampled:sample-X} in Algorithm~\ref{alg:TR-ALS-Sampled} only needs to access $J I_n$ entries of $\Xe$, or less if $\Xe$ is sparse.
This is why our proposed algorithm has a cost which is sublinear in the number of entries in $\Xe$.
Due to Theorem~\ref{thm:main}, each sampled least squares solve on line~\ref{line:tr-als-sampled:ls} has high-probability relative-error guarantees, provided $J$ is large enough.

\subsection{Termination Criteria} \label{sec:termination-criteria}

Both Algorithms~\ref{alg:TR-ALS} and \ref{alg:TR-ALS-Sampled} rely on using some termination criterion.
One such criterion is to terminate the outer loop when the relative error $\|\TR((\Ge^{(n)})_{n=1}^N) - \Xe\|_\F/\|\Xe\|_\F$, or its change, is below some threshold.
Computing the relative error exactly costs $O(I^N R^2)$ which may be too costly for large tensors. 
In particular, it would defeat the purpose of the sublinear complexity of our method.
An alternative is to terminate when the change in $\|\TR((\Ge^{(n)})_{n=1}^N)\|_\F$ is below some threshold.
Such an approach is used for large scale Tucker decomposition by \citet{kolda2008} and \citet{malik2018}.
The norm $\|\TR((\Ge^{(n)})_{n=1}^N)\|_\F$ can be computed efficiently using an algorithm by \citet{mickelin2020}.
When $I_n = I$ and $R_n = R$ for all $n \in [N]$ it has complexity $O(N I R^4)$; see Section~3.3 of \citet{mickelin2020}.
Another approach is to estimate the relative error via sampling, which is used by \citet{battaglino2018} for their randomized CP decomposition method.

\subsection{Adaptive Sample Size} \label{sec:adaptive-sample-size}

Algorithm~\ref{alg:TR-ALS-Sampled} can produce good results even if the sample size $J$ is smaller than what Theorem~\ref{thm:main} suggests.
Choosing an appropriate $J$ for a particular problem can therefore be challenging.
\citet{aggour2020} develop an ALS-based algorithm for CP decomposition with an adaptive sketching rate to speed up convergence.
Similar ideas can easily be incorporated in our Algorithm~\ref{alg:TR-ALS-Sampled}.
Indeed, one of the benefits of our method is that the sample size $J$ can be changed at any point with no overhead cost.
By contrast, for rTR-ALS, another randomized method for TR decomposition which we describe in Section~\ref{sec:experiments}, the entire compression phase would need to be redone if the sketch rate is updated.

\subsection{Complexity Analysis} \label{sec:complexity-analysis}

\begin{table}[h] 
	\centering
	\caption{
		Comparison of leading order computational complexity.
		``$\noiter$'' denotes the number of outer loop iterations in the ALS-based methods.
	} 
	\label{table:complexity}
	\begin{tabular}{ll}
		\toprule
		Method & Complexity \\
		\midrule
		TR-ALS  				& $NIR^2 + \noiter \cdot N I^N R^2$ \\
		rTR-ALS					& $N I^N K + \noiter \cdot N K^N R^2$ \\
		TR-SVD   				& $I^{N+1} + I^N R^3$\\
		TR-SVD-Rand   			& $I^N R^2$ \\
		TR-ALS-S.\ (proposal) 	& $N I R^4 + \noiter \cdot \frac{N I R^{2(N+1)}}{\varepsilon \delta}$ \\
		\bottomrule
	\end{tabular}
\end{table}

Table~\ref{table:complexity} compares the leading order complexity of our proposed method to the four other methods we consider in the experiments.
The other methods are described in Section~\ref{sec:experiments}.
For simplicity, we assume that $I_n = I$ and $R_n = R$ for all $n \in [N]$, and that $N < I$, $R^2 < I$ and $NR < I$.
For rTR-ALS, the intermediate $N$-way Tucker core is assumed to be of size $K \times \cdots \times K$; see \citet{yuan2019} for details.
The complexity for our method in the table assumes $J$ satisfies \eqref{eq:J-bound-2}, and that $\varepsilon$ and $\delta$ are small enough so that \eqref{eq:J-bound-2} simplifies to $J > 4 R^{2N} / (\varepsilon \delta)$.
Any costs for checking convergence criteria are ignored.
This is justified since e.g.\ the termination criterion based on computing $\|\TR((\Ge^{(n)})_{n=1}^N)\|_\F$ described in Section~\ref{sec:termination-criteria} would not impact the complexity of TR-ALS or our TR-ALS-Sampled, and would only impact that of rTR-ALS if $K^N < IR^2$.
The complexities in Table~\ref{table:complexity} are derived in Section~\ref{SUPP:sec:detailed-complexity} of the supplement.

The complexity of our method is sublinear in the number of tensor elements, avoiding the dependence on $I^N$ that the other methods have.
While there is still an exponential dependence on $N$, this is still much better than $I^N$ for low-rank decomposition of large tensors in which case $R \ll I$.
This is particularly true when the input tensor is too large to store in RAM and accessing its elements dominate the cost of the decomposition.
We note that our method typically works well in practice even if \eqref{eq:J-bound-2} is not satisfied.

\section{Experiments} \label{sec:experiments}

\subsection{Decomposition of Real and Synthetic Datasets} \label{sec:experiments-decomposition}

We compare our proposed TR-ALS-Sampled method (Alg.~\ref{alg:TR-ALS-Sampled}) to four other methods: TR-ALS (Alg.~\ref{alg:TR-ALS}),  rTR-ALS (Alg.~1 in \citet{yuan2019}), TR-SVD (Alg.~1 in \citet{mickelin2020}) and a randomized variant of TR-SVD (Alg.~7 in \citet{ahmadi-asl2020}) which we refer to as TR-SVD-Rand.
TR-SVD takes an error tolerance as input and makes the ranks large enough to achieve this tolerance.
To facilitate comparison to the ALS-based methods, we modify TR-SVD so that it takes target ranks as inputs instead.
TR-SVD-Rand takes target ranks as inputs by default. 
For TR-ALS, rTR-ALS and TR-SVD-Rand, we wrote our own implementations since codes were not publicly available. 
For TR-SVD, we modified the function \verb|TRdecomp.m| provided by \citet{mickelin2020}\footnote{Available at \url{https://github.com/oscarmickelin/tensor-ring-decomposition}.}.
As suggested by \citet{ahmadi-asl2020}, we use an oversampling parameter of 10 in TR-SVD-Rand.
The relative error of a decomposition is computed as $\|\TR((\Ge^{(n)})_{n=1}^N) - \Xe\|_\F/\|\Xe\|_\F$.

To get a fair comparison between the different methods we first decompose the input tensor with TR-ALS by running it until a convergence criterion has been met or a maximum number of iterations reached.
All ALS-based algorithms (including TR-ALS) are then applied to the input tensor, running for twice as many iterations as it took the initial TR-ALS run to terminate, without checking any convergence criteria.
This is done to strip away any run time influence that checking convergence has and only compare those aspects that differ between the methods.
TR-ALS-Sampled and rTR-ALS apply sketching in fundamentally different ways.
To best understand the trade-off between run time and accuracy, we start out these two algorithm with small sketch rates ($J$ and $K$, respectively) and increase them until the resulting relative error is less than $(1+\varepsilon)E$, where $E$ is the relative error achieved by TR-ALS during the second run and $\varepsilon \in (0,1)$.
The relative error and run time for the smallest sketch rate that achieve this relative-error bound are then reported.
The exact details on how this is done for each individual experiment is described in Section~\ref{SUPP:sec:additional-experiment-details} of the supplement.

All experiments are run in Matlab R2017a with Tensor Toolbox 2.6 \citep{bader2015} on a computer with an Intel Xeon E5-2630 v3 @ 32x 3.2GHz CPU and 32 GB of RAM.
All our code used in the experiments is available online\footnote{Available at \url{https://github.com/OsmanMalik/tr-als-sampled}.}.

\subsubsection{Randomly Generated Data} \label{sec:experiment-random-data}

Synthetic 3-way tensors are generated by creating 3 cores of size $R' \times I \times R'$ with entries drawn independently from a standard normal distribution.
Here, $R'$ is used to denote true rank, as opposed to the rank $R$ used in the decomposition.
For each core, one entry is chosen uniformly at random and set to 20.
The goal of this is to increase the coherence in the least squares problems, which makes them more challenging for sampling-based methods like ours.
The cores are then combined into a full tensor via \eqref{eq:tr-decomposition-elementwise}.
Finally, Gaussian noise with standard deviation 0.1 is added to all entries in the tensor independently.

\begin{figure}[h!]
	\centering  
	\includegraphics[width=1\columnwidth]{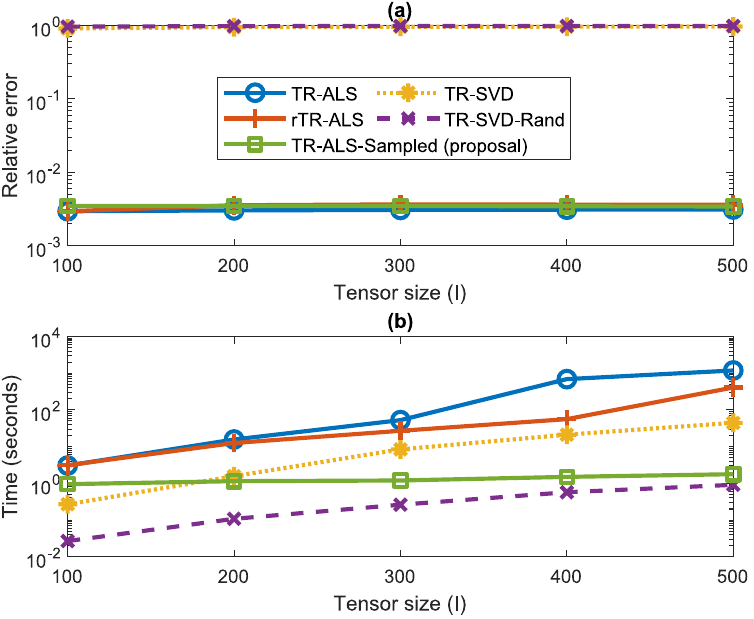}
	\caption{
		Synthetic experiment with true and target ranks $R'=R=10$. 
		(a) Relative error and (b) run time for each of the five methods.
		Average number of ALS iterations used is 21.
	}
	\label{fig:experiment-1a}
\end{figure}

In the first synthetic experiment $R' = 10$ and $I \in \{100, 200, \ldots, 500\}$.
The target rank is also set to $R = 10$. 
Ideally, decompositions should therefore have low error.
Figure~\ref{fig:experiment-1a} shows the relative error and run time for the five different algorithms.
All plotted quantities are the median over 10 runs.
The ALS-based algorithms reach very low errors. 
On average, only 21 iterations are needed, and the number of iterations is fairly stable when the tensor size is increased.
The SVD-based algorithms, by contrast do very poorly, having relative errors close to 1.
Our proposed TR-ALS-Sampled is the fastest method, aside from TR-SVD-Rand which returns very poor results. 
We achieve a substantial speedup of $668\times$ over TR-ALS and $230\times$ over rTR-ALS for $I=500$.

\begin{figure}[h!]
	\centering  
	\includegraphics[width=1\columnwidth]{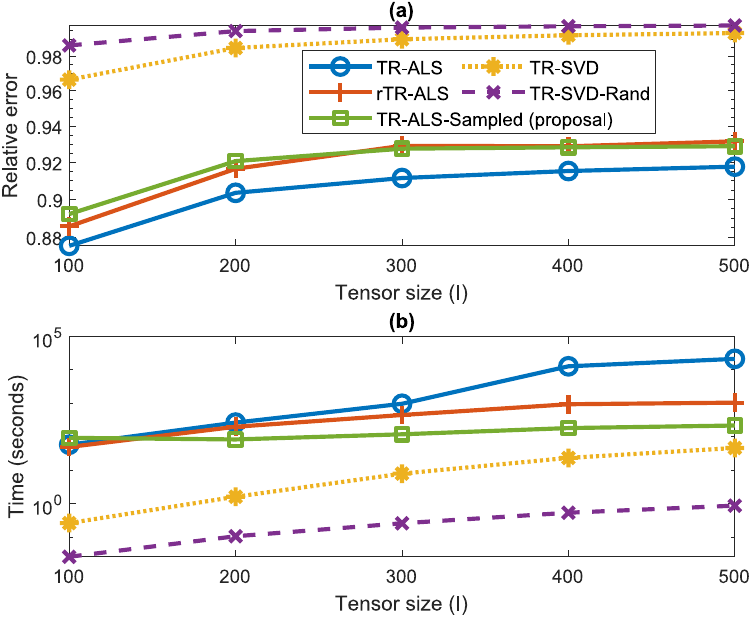}
	\caption{
		Synthetic experiment with true rank $R' = 20$ and target rank $R=10$.
		(a) Relative error and (b) run time for each of the five methods.
		Average number of ALS iterations used is 417.
	}
	\label{fig:experiment-1b}
\end{figure}

In the second synthetic experiment, we set $R' = 20$ but keep the target rank at $R = 10$.
We therefore expect the decompositions to have a larger relative error.
Figure~\ref{fig:experiment-1b} shows the results for these experiments.
All plotted quantities are the median over 10 runs.
The ALS-based methods still outperform the SVD-based ones.
The ALS-based methods now require more iterations, on average running for 417 iterations, with greater variability between runs.
Our proposed TR-ALS-Sampled remains faster than the other ALS-based methods, although with a smaller difference, achieving a speedup of $98\times$ over TR-ALS and $5\times$ over rTR-ALS for $I=500$.
Although the SVD-based methods are faster, especially the randomized variant, they are not useful since they give such a high error.
See Remarks~\ref{SUPP:remark:SVD-performance} and \ref{SUPP:remark:empirical-vs-theoretical-complexity} in the supplement for further discussion of the results in Figures~\ref{fig:experiment-1a} and \ref{fig:experiment-1b}. 

The signal-to-noise ratio (SNR) is about 50 dB and 59 dB, respectively, in the two experiments above.
To validate the robustness of those results to higher levels of noise we run additional experiments on synthetic $300 \times 300 \times 300$ tensors with added Gaussian noise with standard deviation 1 and 10 which corresponds to an SNR of 30 dB and 10 dB, respectively.
All other settings are kept the same as in the first synthetic experiment.
The results, shown in Figure~\ref{fig:higher-snr}, indicate that the earlier experiment results are robust to higher noise levels.
\begin{figure}[ht!]
	\centering  
	\includegraphics[width=1\columnwidth]{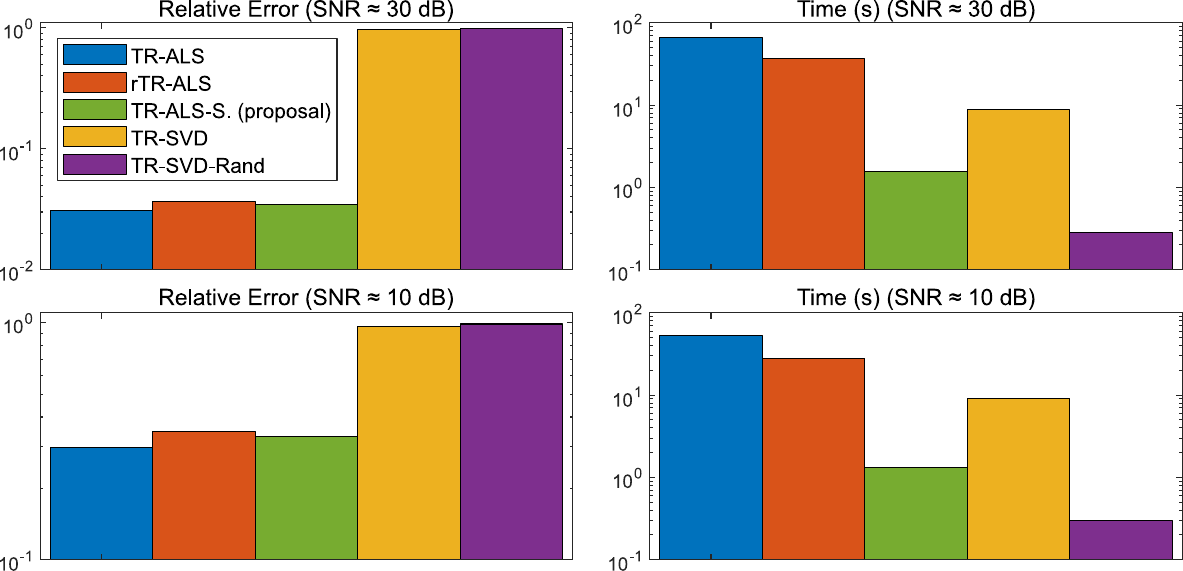}
	\caption{Comparison of methods for higher levels of noise.}
	\label{fig:higher-snr}
\end{figure}

\subsubsection{Highly Oscillatory Functions}

\sisetup{round-mode=places, table-number-alignment=right, table-text-alignment=right}
\begin{table*}[t!]
	\centering
	\caption{Decomposition results for highly oscillatory functions with target rank $R=10$.} 
	\label{tab:rank-10-results-osc-func}
	\begin{tabular}{
			l
			S[round-precision=3, table-figures-decimal=3, table-figures-integer=1]
			S[round-precision=1, table-figures-decimal=1, table-figures-integer=3]
			S[round-precision=3, table-figures-decimal=3, table-figures-integer=1]
			S[round-precision=1, table-figures-decimal=1, table-figures-integer=3]
			S[round-precision=3, table-figures-decimal=3, table-figures-integer=1]
			S[round-precision=1, table-figures-decimal=1, table-figures-integer=3]
		}  
		\toprule
		& \multicolumn{2}{c}{Linear Growth} & \multicolumn{2}{c}{Airy} & \multicolumn{2}{c}{Chirp} \\
		\cmidrule(lr){2-3}
		\cmidrule(lr){4-5}
		\cmidrule(l){6-7}
		Method 				 		& {Error} & {Time (s)} & {Error} & {Time (s)}  & {Error} & {Time (s)}\\
		\midrule
		TR-ALS 				 		& 0.0102  & 370.3  & 0.0203  & 531.7   & 0.0198  & 613.1  \\
		rTR-ALS 			 		& 0.0105  & 382.8  & 0.0186  & 521.6   & 0.0193  & 645.9  \\
		TR-ALS-Sampled (proposal) 	& 0.0109  &   3.4  & 0.0213  &   2.7   & 0.0210  &  36.2  \\
		\bottomrule
	\end{tabular}
\end{table*}

Next, we decompose the highly oscillatory functions considered by \citet{zhao2016}. 
The functions are illustrate in Figure~\ref{fig:oscillatory-func}.
Each function is evaluated at $4^{10}$ points.
These values are then reshaped into 10-way tensors of size $4 \times \cdots \times 4$ which are then decomposed.
Table~\ref{tab:rank-10-results-osc-func} shows the results for when all target ranks are $R=10$.
The SVD-based methods can not be applied in this experiment since they require $R_0 R_1 \leq I_1$ which is not satisfied.

\begin{figure}[h]
	\centering  
	\includegraphics[width=1\columnwidth]{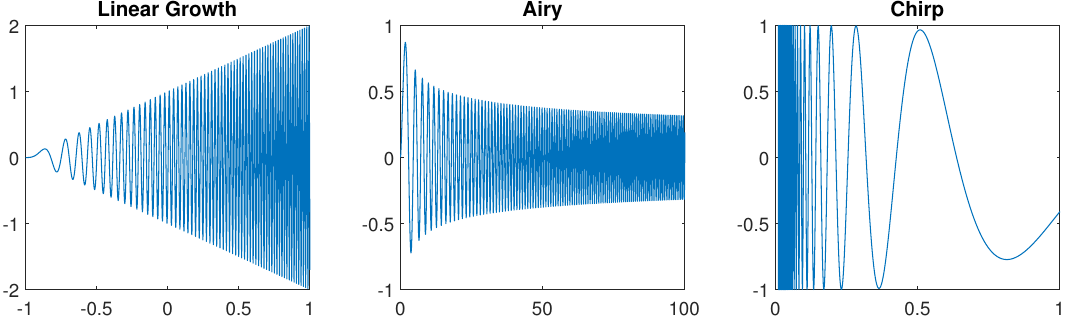}
	\caption{
		The functions are defined as follows. 
		Linear growth: $(x+1) \sin(100(x+1)^2)$.
		Airy: $x^{-1/4} \sin(2 x^{3/2} / 3)$.
		Chirp: $\sin(4/x) \cos(x^2)$.
	}
	\label{fig:oscillatory-func}
\end{figure}

All algorithms achieve a relative error of only 1\%--2\%. 
This is impressive, especially considering that the choice of $R_n = 10$ correspond to a $262\times$ compression rate. 
With an error very similar to that of TR-ALS, our proposal TR-ALS-Sampled achieves speedups of $108 \times$, $193 \times$ and $17 \times$ on the three tensors over TR-ALS.
Our method achieves speedups of $112\times$, $190\times$ and $18\times$ over rTR-ALS.
Notice that rTR-ALS yields no speedup over TR-ALS in Table~\ref{tab:rank-10-results-osc-func}.
The reason for this is that each $I_n = 4$ is so small.
For rTR-ALS to have acceptable accuracy, $K$ must be chosen to be $K = 4$, meaning that rTR-ALS is just as slow as TR-ALS.

\subsubsection{Image and Video Data}

We consider five real image and video datasets\footnote{Links to the datasets are provided in Section~\ref{SUPP:sec:links-to-datasets} of the supplement.} (sizes in parentheses): 
Pavia University ($610 \times 340 \times 103$) and DC Mall ($1280 \times 307 \times 191$) are 3-way tensors containing hyperspectral images.
The first two dimensions are the image height and width, and the third dimension is the number of spectral bands.
Park Bench ($1080 \times 1920 \times 364$) and Tabby Cat ($720 \times 1280 \times 286$) are 3-way tensors representing grayscale videos of a man sitting on a park bench and a cat, respectively.
The first two dimension are frame height and width, and the third dimension is the number of frames.
Red Truck ($128 \times 128 \times 3 \times 72$) is a 4-way tensor consisting of 72 color images of size 128 by 128 pixels depicting a red truck from different angles.
It is a subset of the COIL-100 dataset \citep{nene1996}.

Tables~\ref{tab:rank-10-results} and \ref{tab:rank-20-results} show results for when all target ranks are $R = 10$ and $R = 20$, respectively.
The SVD-based algorithms require $R_0 R_1 \leq I_1$, so they cannot handle the Red Truck tensor when $R=20$ since $I_1 = 128$ for that dataset. 
All results are based on a single trial.
Our proposed TR-ALS-Sampled runs faster than the other ALS-based methods on all image and video datasets, achieving up to $1518\times$ speedup over TR-ALS (for the Park Bench tensor with $R=10$) and a $21\times$ speedup over rTR-ALS (for the Red Truck tensor with $R=20$).
The SVD-based methods perform better than they did on the randomly generated data, but remain worse than the ALS-based methods.
TR-SVD-Rand is always the fastest method, but also has a substantially higher error.

\sisetup{round-mode=places, table-number-alignment=right, table-text-alignment=right}

\begin{table*}[t!]
	\centering
	\caption{Decomposition results for real datasets with target rank $R=10$.
	Time is in seconds.} 
	\label{tab:rank-10-results}
	\begin{tabular}{
			l
			S[round-precision=2, table-figures-decimal=2, table-figures-integer=1]
			S[round-precision=1, table-figures-decimal=1, table-figures-integer=2]
			S[round-precision=2, table-figures-decimal=2, table-figures-integer=1]
			S[round-precision=1, table-figures-decimal=1, table-figures-integer=3]
			S[round-precision=2, table-figures-decimal=2, table-figures-integer=1]
			S[round-precision=1, table-figures-decimal=1, table-figures-integer=4]
			S[round-precision=2, table-figures-decimal=2, table-figures-integer=1]
			S[round-precision=1, table-figures-decimal=1, table-figures-integer=4]
			S[round-precision=2, table-figures-decimal=2, table-figures-integer=1]
			S[round-precision=1, table-figures-decimal=1, table-figures-integer=3]
		}  
		\toprule
		& \multicolumn{2}{c}{Pavia Uni.} & \multicolumn{2}{c}{DC Mall} & \multicolumn{2}{c}{Park Bench} & \multicolumn{2}{c}{Tabby Cat} & \multicolumn{2}{c}{Red Truck}\\
		\cmidrule(lr){2-3}
		\cmidrule(lr){4-5}
		\cmidrule(lr){6-7}
		\cmidrule(lr){8-9}
		\cmidrule(l){10-11}
		Method 				 & {Error} & {Time} & {Error} & {Time}  & {Error} & {Time}  & {Error} & {Time}  & {Error} & {Time}  \\
		\midrule
		TR-ALS 				 &  0.1639 &  62.075 & 0.15981 &   801.4 & 0.073525 & 5384.9 & 0.13579 & 1769.4 &  0.1765 &   181.21 \\
		rTR-ALS 			 & 0.17566 &  16.251 & 0.17346 &   25.37 &  0.08053 & 64.814 & 0.14447 & 8.9931 & 0.19307 &   21.984 \\
		TR-ALS-S. (proposal) & 0.17755 &   1.727 &  0.1754 &  3.0827 & 0.080136 & 3.5475 & 0.14786 & 2.4187 & 0.19077 &   1.9332 \\
		\midrule
		TR-SVD 				 &  0.2121 &   2.138 & 0.20415 &  12.773 &  0.07814 & 136.45 & 0.15359 &  37.73 & 0.26219 &   1.2844 \\
		TR-SVD-Rand 		 & 0.32651 & 0.16243 & 0.29738 & 0.33525 &  0.11089 & 2.7777 & 0.19614 & 1.5828 & 0.29276 & 0.085354 \\
		\bottomrule
	\end{tabular}
\end{table*}

\begin{table*}[t!]
	\centering
	\caption{
		Decomposition results for real datasets with target rank $R=20$.
		The \xmark\ signifies that the SVD-based methods cannot handle this case since they require $R_0 R_1 \leq I_1$.
		Time is in seconds.
	} 
	\label{tab:rank-20-results}
	\begin{tabular}{
			l
			S[round-precision=2, table-figures-decimal=2, table-figures-integer=1]
			S[round-precision=1, table-figures-decimal=1, table-figures-integer=3]
			S[round-precision=2, table-figures-decimal=2, table-figures-integer=1]
			S[round-precision=1, table-figures-decimal=1, table-figures-integer=3]
			S[round-precision=2, table-figures-decimal=2, table-figures-integer=1]
			S[round-precision=1, table-figures-decimal=1, table-figures-integer=4]
			S[round-precision=2, table-figures-decimal=2, table-figures-integer=1]
			S[round-precision=1, table-figures-decimal=1, table-figures-integer=4]
			S[round-precision=2, table-figures-decimal=2, table-figures-integer=1]
			S[round-precision=1, table-figures-decimal=1, table-figures-integer=4]
		}
		\toprule
		& \multicolumn{2}{c}{Pavia Uni.} & \multicolumn{2}{c}{DC Mall} & \multicolumn{2}{c}{Park Bench} & \multicolumn{2}{c}{Tabby Cat} & \multicolumn{2}{c}{Red Truck}\\
		\cmidrule(lr){2-3}
		\cmidrule(lr){4-5}
		\cmidrule(lr){6-7}
		\cmidrule(lr){8-9}
		\cmidrule(l){10-11}
		Method 				 & {Error} & {Time} & {Error} & {Time}  & {Error} & {Time}  & {Error} & {Time}  & {Error} & {Time}  \\
		\midrule
		TR-ALS 				 & 0.061195 &  249.91 & 0.057176 & 622.81 & 0.041872 & 2385.0 & 0.10838 & 1397.2 & 0.11019 & 1279.1 \\
		rTR-ALS 			 & 0.066827 &  194.61 & 0.062014 & 349.63 & 0.045118 & 531.46 & 0.11564 & 92.591 & 0.11864 & 470.59 \\
		TR-ALS-S. (proposal) &  0.06652 &  13.738 & 0.062687 & 26.535 & 0.045891 & 24.952 & 0.11909 &   19.0 & 0.11958 & 22.607 \\
		\midrule
		TR-SVD 				 & 0.098693 &   4.587 &  0.10265 & 25.401 & 0.048204 & 271.99 & 0.13057 &  110.0 &  \xmark & \xmark \\
		TR-SVD-Rand 		 &  0.28156 & 0.50099 &  0.25283 & 1.0098 & 0.098901 & 8.5071 & 0.16821 & 4.0938 &  \xmark & \xmark \\
		\bottomrule
	\end{tabular}
\end{table*}

A popular preprocessing step used in tensor decomposition is to first reshape a tensor into a tensor with more modes.
Next, we therefore try reshaping the five tensors above so that they each have twice as many modes.
Some of the dimensions are truncated slightly to allow for this reshaping.
Our method yields speedups in the range $94\times$--$2880\times$ over TR-ALS and $73\times$--$1484\times$ over rTR-ALS.
The SVD-based methods cannot be applied since they require $R_0 R_1 \leq I_1$ which is not satisfied for these reshaped tensors.
The speed benefit of our method appears to be \emph{even greater} when increasing the number of modes.
Please see Section~\ref{SUPP:sec:additional-details-image-and-video} of the supplement for further details on this experiment, including detailed results.

An alternative to sampling according to the distribution defined by the leverage scores in \eqref{eq:p-definition} and \eqref{eq:q-definition} is to simply sample according to the uniform distribution, i.e.,  $\qbf^{\neq n} = 1/\prod_{j \neq n} I_j$.
Although such a sampling approach does not come with any guarantees, it is faster than leverage score sampling since it avoids the cost of computing the sampling distribution.
In order to investigate the performance of uniform sampling, we repeat the experiments whose results are show in Tables~\ref{tab:rank-10-results} and \ref{tab:rank-20-results} for TR-ALS-Sampled, but using uniform sampling instead of leverage score sampling.
Everything else is kept the same.
Table~\ref{tab:uniform-sampling} shows the results.

\begin{table}[ht!]
	\centering
	\caption{
		Decomposition results for TR-ALS-Sampled with \emph{uniform sampling}.
	} 
	\label{tab:uniform-sampling}
	\begin{tabular}{lrrrr}  
		\toprule
		& \multicolumn{2}{c}{$R=10$} & \multicolumn{2}{c}{$R=20$} \\
		\cmidrule(lr){2-3} 	
		\cmidrule(l){4-5}
		Dataset 	& Error & Time (s) & Error & Time (s) \\
		\midrule
		Pavia Uni. 	& 0.18 & 1.3 & 0.07 & 9.0  \\
		DC Mall 	& 0.18 & 3.7 & 0.06 & 22.2 \\
		Park Bench 	& 0.08 & 5.8 & 0.05 & 28.0 \\
		Tabby Cat 	& 0.15 & 2.2 & 0.12 & 23.1 \\
		Red Truck 	& 0.20 & 1.5 & 0.12 & 18.7 \\
		\bottomrule
	\end{tabular}
\end{table}

Using uniform instead of leverage score sampling does not appear to impact the error and overall run time much.
However, uniform sampling requires 4.6--8.8 times as many samples to reach the target accuracy.
Although more samples increase the cost of the least squares solve, the cost of updating the sampling distribution is avoided.
It may be possible to combine uniform and leverage score sampling to speed up our method while still retaining some guarantees.
Uniform sampling has also yielded promising empirical results for the CP decomposition; see e.g.\ \citet{battaglino2018} and \citet{aggour2020}.

\subsection{Rapid Feature Extraction for Classification} \label{sec:experiments-feature-extraction}

Inspired by experiments in \citet{zhao2016}, we now give a practical example of how our method can be used for rapid feature extraction for classification.
We consider the full COIL-100 dataset which consists of 7200 color images of size $128 \times 128$ depicting 100 different objects from 72 different angles each.
The images are downsampled to $32 \times 32$ and stacked into a tensor $\Xe$ of size $32 \times 32 \times 3 \times 7200$.
We then use the various TR decomposition algorithms to compute a TR decomposition of $\Xe$ with each rank $R_n = 5$.
Additional details on algorithm settings are given in Section~\ref{SUPP:sec:additional-details-feature-extraction} of the supplement.
The TR core $\Ge^{(4)}$ is of size $5 \times 7200 \times 5$ and contains latent features for the 4th mode of $\Xe$.
We reshape it into a $7200 \times 25$ feature matrix and apply the $k$-nearest neighbor algorithm with $k = 1$.
Table~\ref{tab:feature-extraction} reports the decomposition error as well as the accuracy achieved when classifying the 7200 images into the 100 different classes.
The classification is done using 10-fold cross validation.
The ALS-based methods yield a lower error than the SVD-based methods.
All methods result in a similar classification accuracy. 
The higher decomposition errors for the SVD-based methods do not seem to compromise classification accuracy.

\begin{table}[ht!]
	\centering
	\caption{Decomposition error and classification accuracy for rapid feature extraction experiment.} 
	\label{tab:feature-extraction}
	\begin{tabular}{lrrr}  
		\toprule
		Method 						& Time (s) 	& Error	& Acc.\ (\%)	\\
		\midrule
		TR-ALS 						& 248.7		& 0.27	& 99.07 	 	\\
		rTR-ALS 					& 3.1		& 0.29	& 98.96 	 	\\
		TR-ALS-S.\ (proposal) 		& 12.7		& 0.28	& 99.32 	 	\\
		TR-SVD 						& 7.3   	& 0.37	& 99.68 	 	\\
		TR-SVD-Rand					& 0.7		& 0.33	& 99.85 	 	\\
		\bottomrule
	\end{tabular}
\end{table}

\section{Conclusion} \label{sec:conclusion}

We have proposed a method for TR decomposition which uses leverage score sampled ALS and has complexity sublinear in the number of input tensor entries.
We also proved high-probability relative-error guarantees for the sampled least squares problems.
Our method achieved substantial speedup over competing methods in experiments on both synthetic and real data, in some cases by as much as two or three orders of magnitude, while maintaining good accuracy.
We also demonstrated how our method can be used for rapid feature extraction.

The datasets in our experiments are dense.
Based on limited testing, rTR-ALS seems to do particularly well on sparse datasets, outperforming our proposed method. 
Our method still yielded a substantial speedup over TR-ALS.
Further investigating this, and optimizing the various methods for sparse tensors, is an interesting direction for future research.

\section*{Acknowledgments}

We would like to thank the reviewers for their many helpful comments and suggestions which helped improve this paper.
The work of OAM was supported by the AFOSR grant FA9550-20-1-0138.

\bibliography{library}

\onecolumn

\title{Supplementary Material}
\date{}

\maketitle
\thispagestyle{empty}

\section{Missing Proofs} \label{SUPP:sec:proofs}

In this section we give proofs for Lemma~\ref{lemma:leverage-score-sampling} and Theorem~\ref{thm:main} in the main manuscript.
We first state some results that we will use in these proofs.

Lemma~\ref{SUPP:lemma:structural-conds} is a variant of Lemma~1 by \citet{drineas2011} but for multiple right hand sides.
The proof by \citet{drineas2011} remains essentially identical with only minor modifications to account for the multiple right hand sides.
\begin{lemma} \label{SUPP:lemma:structural-conds}
	Let $\OPT \defeq \min_{\Xbf} \| \Abf \Xbf - \Ybf\|_\F$ be a least squares problem where $\Abf \in \Rb^{I \times R}$ and $I > R$, and let $\Ubf \in \Rb^{I \times \rank(\Abf)}$ contain the left singular vectors of $\Abf$. 
	Moreover, let $\Ubf^\perp$ be an orthogonal matrix whose columns span the space perpendicular to $\range(\Ubf)$ and define $\Ybf^\perp \defeq \Ubf^\perp (\Ubf^{\perp})^\top \Ybf$.
	Let $\Sbf \in \Rb^{J \times I}$ be a matrix satisfying
	\begin{equation} \label{SUPP:eq:cond-1}
		\sigma^2_{\min} (\Sbf \Ubf) \geq \frac{1}{\sqrt{2}},
	\end{equation}
	\begin{equation} \label{SUPP:eq:cond-2}
		\| \Ubf^\top \Sbf^\top \Sbf \Ybf^\perp\|_\F^2 \leq \frac{\varepsilon}{2} \OPT^2,
	\end{equation}
	for some $\varepsilon \in (0,1)$.
	Then, it follows that
	\begin{equation}
		\| \Abf \tilde{\Xbf} - \Ybf \|_\F \leq (1+\varepsilon)\OPT,
	\end{equation}
	where $\tilde{\Xbf} \defeq \argmin_{\Xbf} \| \Sbf \Abf \Xbf - \Sbf \Ybf \|_\F$.
\end{lemma}

Lemma~\ref{SUPP:lemma:singular-value-bound} is a slight restatement of Theorem~2.11 in the monograph by \citet{woodruff2014}.
The statement in that monograph has a constant $144$ instead of $8/3$.
However, we found that $8/3$ is sufficient under the assumption that $\varepsilon \in (0,1)$.
The proof given by \citet{woodruff2014} otherwise remains the same.  
\begin{lemma} \label{SUPP:lemma:singular-value-bound}
	Let $\Abf \in \Rb^{I \times R}$, $\varepsilon \in (0,1)$, $\eta \in (0,1)$ and $\beta \in (0,1]$.
	Suppose
	\begin{equation} \label{SUPP:eq:J-bound}
		J > \frac{8}{3} \frac{R \ln (2R/\eta)}{\beta \varepsilon^2}
	\end{equation} 
	and that $\Sbf \sim \Dc(J, \qbf)$ is a leverage score sampling matrix for $(\Abf, \beta)$ (see Definition~\ref{def:leverage-score-sampling-matrix} in the main manuscript).
	Then, with probability at least $1 - \eta$, the following holds:
	\begin{equation}
		(\forall i \in [\rank(\Abf)]) \;\;\;\; 1-\varepsilon \leq \sigma_i^2(\Sbf \Ubf) \leq 1+\varepsilon,
	\end{equation}
	where $\Ubf \in \Rb^{I \times \rank(\Abf)}$ contains the left singular vectors of $\Abf$.
\end{lemma}

Lemma~\ref{SUPP:lemma:approx-mm} is a part of Lemma~8 by \citet{drineas2006b}.
\begin{lemma} \label{SUPP:lemma:approx-mm}
	Let $\Abf$ and $\Bbf$ be matrices with $I$ rows, and let $\beta \in (0, 1]$. 
	Let $\qbf \in \Rb^I$ be a probability distribution satisfying
	\begin{equation} \label{SUPP:eq:approx-mm-lemma-cond}
		\qbf(i) \geq \beta \frac{\|\Abf(i, :)\|^2_2}{\| \Abf \|_\F^2} \;\;\;\; \text{for all } i \in [I].
	\end{equation}
	If $\Sbf \sim \Dc(J, \qbf)$, then
	\begin{equation}
		\Eb \|\Abf^\top \Bbf - \Abf^\top \Sbf^\top \Sbf \Bbf \|_\F^2 \leq \frac{1}{\beta J} \|\Abf\|_\F^2 \|\Bbf\|_\F^2.
	\end{equation}
\end{lemma}

In the following, $\otimes$ denotes the matrix Kronecker product; see e.g. Section~1.3.6 of \citet{golub2013} for details.
\begin{lemma} \label{SUPP:lemma:subchain-range}
	For $n \in [N]$, the subchain $\Ge^{\neq n}$ satisfies
	\begin{equation} \label{SUPP:eq:lemma-range}
	\range(\Gbf_{[2]}^{\neq n}) \subseteq \range(\Gbf_{(2)}^{(n-1)} \otimes \cdots \otimes \Gbf_{(2)}^{(1)} \otimes \Gbf_{(2)}^{(N)} \otimes \cdots \otimes \Gbf_{(2)}^{(n+1)}).
	\end{equation}
\end{lemma}
\begin{proof}
	We have
	\begin{equation}
	\begin{aligned}
		&\Gbf_{[2]}^{\neq n}(\overline{i_{n+1} \cdots i_{N} i_1 \cdots i_{n-1}}, \overline{r_{n-1} r_{n}}) = \\
		&\sum_{\substack{r_1, \ldots, r_{n-2}\\r_{n+1}, \ldots, r_N}} \Gbf_{(2)}^{(n+1)}(i_{n+1}, \overline{r_{n} r_{n+1}}) \cdots \Gbf_{(2)}^{(N)}(i_{N}, \overline{r_{N-1} r_{N}}) \Gbf_{(2)}^{(1)}(i_{1}, \overline{r_{N} r_{1}}) \cdots \Gbf_{(2)}^{(n-1)}(i_{n-1}, \overline{r_{n-2} r_{n-1}}).  
	\end{aligned}
	\end{equation}
	From this, it follows that
	\begin{equation}
	\begin{aligned}
		&\Gbf_{[2]}^{\neq n}(:, \overline{r_{n-1} r_{n}}) = \\
		&\sum_{\substack{r_1, \ldots, r_{n-2}\\r_{n+1}, \ldots, r_N}} \Gbf_{(2)}^{(n-1)}(:, \overline{r_{n-2} r_{n-1}}) \otimes \cdots \otimes  \Gbf_{(2)}^{(1)}(:, \overline{r_{N} r_{1}}) \otimes \Gbf_{(2)}^{(N)}(:, \overline{r_{N-1} r_{N}}) \otimes \cdots \otimes \Gbf_{(2)}^{(n+1)}(:, \overline{r_{n} r_{n+1}}).  
	\end{aligned}
	\end{equation}
	The right hand side of this equation is a sum of columns in
	\begin{equation} \label{SUPP:eq:kronecker-product}
		\Gbf_{(2)}^{(n-1)} \otimes \cdots \otimes \Gbf_{(2)}^{(1)} \otimes \Gbf_{(2)}^{(N)} \otimes \cdots \otimes \Gbf_{(2)}^{(n+1)}.
	\end{equation}
	Consequently, every column of $\Gbf_{[2]}^{\neq n}$ is in the range of the matrix in \eqref{SUPP:eq:kronecker-product}, and the claim in \eqref{SUPP:eq:lemma-range} follows.
\end{proof}

\begin{lemma} \label{SUPP:lemma:leverage-score-inequality}
	Let $\Abf$ and $\Bbf$ be two matrices with $I$ rows such that $\range(\Abf) \subseteq \range(\Bbf)$. 
	Then $\ell_i(\Abf) \leq \ell_i(\Bbf)$ for all $i \in [I]$.
\end{lemma}
\begin{proof}
	Let $\Qbf \in \Rb^{I \times \rank(\Bbf)}$ be an orthogonal matrix containing the left singular vectors of $\Bbf$. 
	Then there exists a matrix $\Mbf$ such that $\Abf = \Qbf \Mbf$. 
	Let $\Ubf \Sigmabf \Vbf^\top = \Mbf$ be the thin SVD of $\Mbf$ (i.e., such that $\Ubf$ and $\Vbf$ have only $\rank(\Mbf)$ columns and $\Sigmabf \in \Rb^{\rank(\Mbf) \times \rank(\Mbf)}$).
	Then $\Abf = \Qbf \Ubf \Sigmabf \Vbf^\top = \Wbf \Sigmabf \Vbf^\top$, where $\Wbf \defeq \Qbf \Ubf$. 
	Since
	\begin{equation}
	\Wbf^\top \Wbf = \Ubf^\top \Qbf^\top \Qbf \Ubf = \Ibf,
	\end{equation}
	$\Wbf$ is orthogonal, so $\Wbf \Sigmabf \Vbf^\top = \Abf$ is a thin SVD of $\Abf$.
	It follows that
	\begin{equation}
	\ell_i(\Abf) = \|\Wbf(i,:)\|_2^2 = \| \Qbf(i,:) \Ubf \|_2^2 \leq \| \Qbf(i,:) \|_2^2 \| \Ubf \|^2_2 = \| \Qbf(i,:) \|_2^2 = \ell_i(\Bbf).
	\end{equation}
\end{proof}

\begin{remark} \label{SUPP:remark:nonzero-cores}
	It is reasonable to assume that $\Ge^{\neq n} \neq \zerobf$ for all $n \in [N]$ during the execution of Algorithms~\ref{alg:TR-ALS} and \ref{alg:TR-ALS-Sampled}.
	If this was not the case, the least squares problem for the $n$th iteration of the inner for loop in these algorithms would have a zero system matrix which would make the least squares problem meaningless.
	The assumption $\Ge^{\neq n} \neq \zerobf$ for all $n \in [N]$ also implies that $\Ge^{(n)} \neq \zerobf$ for all $n \in [N]$.
	This means that $\rank(\Gbf_{(2)}^{(n)}) \geq 1$ and $\rank(\Gbf_{[2]}^{\neq n}) \geq 1$ for all $n \in [N]$, so that the various divisions with these quantities in this paper are well-defined.
\end{remark}

Lemma~\ref{SUPP:lemma:leverage-score-sampling} is a restatement of Lemma~\ref{lemma:leverage-score-sampling} in the main manuscript.
\begin{lemma} \label{SUPP:lemma:leverage-score-sampling}
	Let $\beta_n$ be defined as
	\begin{equation}
	\beta_n \defeq \Big( R_{n-1} R_n \prod_{\substack{j=1\\j \notin \{n-1,n\}}}^N R_j^2 \Big)^{-1}.
	\end{equation}
	For each $n \in [N]$, the vector $\qbf^{\neq n}$ is a probability distribution on $[\prod_{j \neq n} I_j]$, and $\Sbf \sim \Dc(J, \qbf^{\neq n})$ is a leverage score sampling matrix for $(\Gbf_{[2]}^{\neq n}, \beta_n)$.
\end{lemma}
\begin{proof}
	All $\qbf^{\neq n}(i)$ are clearly nonnegative.
	Moreover,
	\begin{equation}
		\sum_{\substack{i_1, \ldots, i_{n-1}\\i_{n+1}, \ldots, i_N}} \qbf^{\neq n}(\overline{i_{n+1} \cdots i_N i_1 \cdots i_{n-1}}) = \prod_{\substack{j=1\\j \neq n}}^N \sum_{i_j} \pbf^{(j)}(i_j) = \prod_{\substack{j=1\\j \neq n}}^N 1 = 1,
	\end{equation}
	since each $\pbf^{(j)}$ is a probability distribution. 
	So $\qbf^{\neq n}$ is clearly also a probability distribution.
	Next, define the vector $\pbf \in \Rb^{\prod_{j \neq n} I_n}$ via
	\begin{equation}
		\pbf(i) \defeq \frac{\ell_i(\Gbf_{[2]}^{\neq n})}{\rank(\Gbf_{[2]}^{\neq n})}.
	\end{equation}
	To show that $\Sbf$ is a leverage score sampling matrix for $(\Gbf_{[2]}^{\neq n}, \beta_n)$, we need to show that $\qbf^{\neq n} (i) \geq \beta_n \pbf(i)$ for all $i = \overline{i_{n+1} \cdots i_N i_1 \cdots i_{n-1}} \in [\prod_{j \neq n} I_n]$.
	To this end, combine Lemmas~\ref{SUPP:lemma:subchain-range} and \ref{SUPP:lemma:leverage-score-inequality} to get
	\begin{equation} \label{SUPP:eq:leverage-score-inequality}
		\ell_i(\Gbf_{[2]}^{\neq n}) \leq \ell_i(\Gbf_{(2)}^{(n-1)} \otimes \cdots \otimes \Gbf_{(2)}^{(1)} \otimes \Gbf_{(2)}^{(N)} \otimes \cdots \otimes \Gbf_{(2)}^{(n+1)}).
	\end{equation}
	For each $n \in [N]$, let $\Ubf^{(n)} \in \Rb^{I_n \times \rank(\Gbf_{(2)}^{(n)})}$ contain the left singular vectors of $\Gbf_{(2)}^{(n)}$. 
	It is well-known \citep{vanloan2000} that the matrix
	\begin{equation}
		\Ubf^{(n-1)} \otimes \cdots \otimes \Ubf^{(1)} \otimes \Ubf^{(N)} \otimes \cdots \otimes \Ubf^{(n+1)}
	\end{equation}
	contains the left singular vectors corresponding to nonzero singular values of
	\begin{equation}
		\Gbf_{(2)}^{(n-1)} \otimes \cdots \otimes \Gbf_{(2)}^{(1)} \otimes \Gbf_{(2)}^{(N)} \otimes \cdots \otimes \Gbf_{(2)}^{(n+1)}.
	\end{equation}
	Consequently, 
	\begin{equation} \label{SUPP:eq:leverage-score-kronecker}
	\begin{aligned}
		&\ell_i(\Gbf_{(2)}^{(n-1)} \otimes \cdots \otimes \Gbf_{(2)}^{(1)} \otimes \Gbf_{(2)}^{(N)} \otimes \cdots \otimes \Gbf_{(2)}^{(n+1)}) \\
		&= \big( (\Ubf^{(n-1)} \otimes \cdots \otimes \Ubf^{(1)} \otimes \Ubf^{(N)} \otimes \cdots \otimes \Ubf^{(n+1)}) (\Ubf^{(n-1)} \otimes \cdots \otimes \Ubf^{(1)} \otimes \Ubf^{(N)} \otimes \cdots \otimes \Ubf^{(n+1)})^\top \big)_{ii} \\
		&= \big((\Ubf^{(n-1)} \Ubf^{(n-1)\top}) \otimes \cdots \otimes (\Ubf^{(1)} \Ubf^{(1)\top}) \otimes (\Ubf^{(N)} \Ubf^{(N)\top}) \otimes \cdots \otimes (\Ubf^{(n+1)} \Ubf^{(n+1)\top})\big)_{ii} \\
		&= (\Ubf^{(n-1)} \Ubf^{(n-1)\top})_{i_{n-1} i_{n-1}} \cdots (\Ubf^{(1)} \Ubf^{(1)\top})_{i_1 i_1} (\Ubf^{(N)} \Ubf^{(N)\top})_{i_N i_N} \cdots (\Ubf^{(n+1)} \Ubf^{(n+1)\top})_{i_{n+1} i_{n+1}} \\
		&= \ell_{i_{n-1}}(\Gbf_{(2)}^{(n-1)}) \cdots \ell_{i_{1}}(\Gbf_{(2)}^{(1)}) \ell_{i_{N}}(\Gbf_{(2)}^{(N)}) \cdots \ell_{i_{n+1}}(\Gbf_{(2)}^{(n+1)}),
	\end{aligned}
	\end{equation}
	where the first and fourth equalities follow from the definition of leverage score and the fact that 
	\begin{equation}
		\|\Mbf(i,:)\|_2^2 = (\Mbf \Mbf^\top)(i,i)
	\end{equation}
	for any matrix $\Mbf$, and the second equality follows from well-known properties of the Kronecker product \citep{vanloan2000}.
	Combining \eqref{SUPP:eq:leverage-score-inequality} and \eqref{SUPP:eq:leverage-score-kronecker}, we have
	\begin{equation} \label{SUPP:eq:leverage-score-compact-ineq}
		\ell_i(\Gbf_{[2]}^{\neq n}) \leq \prod_{\substack{j=1\\j \neq n}}^N \ell_{i_j}(\Gbf_{(2)}^{(j)}).	
	\end{equation}
	We therefore have
	\begin{equation}
		\qbf^{\neq n}(i) = \frac{\prod_{j \neq n} \ell_{i_j} (\Gbf_{(2)}^{(j)})}{\prod_{j \neq n} \rank(\Gbf_{(2)}^{(j)})} \geq \frac{\ell_i(\Gbf_{[2]}^{\neq n})}{R_{n-1} R_n \prod_{j \in [N] \setminus \{n-1, n\}} R_j^2} = \beta_n \ell_i(\Gbf_{[2]}^{\neq n}) \geq \beta_n \pbf(i)
	\end{equation}
	as desired, where the first inequality follows from \eqref{SUPP:eq:leverage-score-compact-ineq} and the fact that $\rank(\Gbf_{(2)}^{(j)}) \leq R_{j-1} R_j$.
\end{proof}

The following is a restatement of Theorem~\ref{thm:main} in the main manuscript. 
The proof is similar to that of Theorem~2 by \citet{drineas2011}.
\begin{theorem} \label{SUPP:thm:main}
	Let $\Sbf \sim \Dc(J, \qbf^{\neq n})$, $\varepsilon \in (0,1)$, $\delta \in (0,1)$ and $\tilde{\Ze} \defeq \argmin_{\Ze} \| \Sbf \Gbf_{[2]}^{\neq n} \Zbf_{(2)}^\top - \Sbf \Xbf_{[n]}^\top \|_\F$.
	If 
	\begin{equation}
		J > \Big( \prod_{j=1}^N R_j^2 \Big) \max \Big( \frac{16}{3(\sqrt{2} - 1)^2} \ln\Big(\frac{4 R_{n-1} R_n}{\delta}\Big), \frac{4}{\varepsilon\delta} \Big),
	\end{equation}
	then the following holds with probability at least $1-\delta$:
	\begin{equation} \label{SUPP:eq:main-statement}
		\| \Gbf_{[2]}^{\neq n} \tilde{\Zbf}_{(2)}^\top - \Xbf_{[n]}^\top \|_\F \leq (1+\varepsilon) \min_{\Ze} \| \Gbf_{[2]}^{\neq n} \Zbf_{(2)}^\top - \Xbf_{[n]}^\top \|_\F.
	\end{equation}
\end{theorem}
\begin{proof}
	Let $\Ubf \in \Rb^{\prod_{j \neq n} I_j \times \rank(\Gbf_{[2]}^{\neq n})}$ contain the left singular vectors of $\Gbf_{[2]}^{\neq n}$.
	According to Lemma~\ref{SUPP:lemma:leverage-score-sampling}, $\Sbf$ is a leverage score sampling matrix for $(\Gbf_{[2]}^{\neq n}, \beta_n)$.
	Since $\Ubf$ has at most $R_{n-1} R_n$ columns and
	\begin{equation}
		J > \frac{16}{3 (\sqrt{2}-1)^2} \Big( \prod_{j = 1}^N R_j^2 \Big) \ln \Big(\frac{4 R_{n-1} R_n}{\delta} \Big),
	\end{equation}
	choosing $\varepsilon = 1 - 1/\sqrt{2}$ and $\eta = \delta/2$ in Lemma~\ref{SUPP:lemma:singular-value-bound} therefore gives that 
	\begin{equation} \label{SUPP:eq:cond-1-pf}
		\sigma_{\min}^2(\Sbf\Ubf) \geq 1/\sqrt{2}
	\end{equation}
	with probability at least $1-\delta/2$. 
	Similarly to Lemma~\ref{SUPP:lemma:structural-conds}, define $(\Xbf_{[n]}^\top)^\perp \defeq \Ubf^\perp (\Ubf^\perp)^\top \Xbf_{[n]}^\top$ and $\OPT \defeq \min_{\Ze} \| \Gbf_{[2]}^{\neq n} \Zbf_{(2)}^\top - \Xbf_{[n]}^\top \|_\F = \| (\Xbf_{[n]}^\top)^\perp \|_\F$.
	Since 
	\begin{equation}
		\qbf^{\neq n}(i) \geq \beta_n \frac{\ell_i(\Gbf_{[2]}^{\neq n})}{\rank(\Gbf_{[2]}^{\neq n})} = \beta_n \frac{\| \Ubf(i, :) \|_2^2}{\|\Ubf\|_\F^2} \;\;\;\; \text{for all } i \in [I],
	\end{equation}
	and $\Ubf^\top (\Xbf_{[n]}^\top)^\perp = \zerobf$, Lemma~\ref{SUPP:lemma:approx-mm} gives that 
	\begin{equation}
		\Eb \| \Ubf^\top \Sbf^\top \Sbf (\Xbf_{[n]}^\top)^\perp \|_\F^2 \leq \frac{1}{\beta_n J} \|\Ubf\|_\F^2 \| (\Xbf_{[n]}^\top)^\perp \|_\F^2 \leq \frac{R_{n-1} R_n}{\beta_n J} \OPT^2.
	\end{equation}
	By Markov's inequality,
	\begin{equation}
		\Pb(\| \Ubf^\top \Sbf^\top \Sbf (\Xbf_{[n]}^\top)^\perp \|_\F^2 
		> \varepsilon \OPT^2/2) \leq \frac{\Eb \| \Ubf^\top \Sbf^\top \Sbf (\Xbf_{[n]}^\top)^\perp \|_\F^2}{\varepsilon \OPT^2/2} \leq \frac{2}{\varepsilon J} \Big( \prod_{j \in [N]} R_j^2 \Big) < \frac{\delta}{2}
	\end{equation}
	since 
	\begin{equation}
		J > \frac{4}{\varepsilon\delta} \Big( \prod_{j=1}^N R_j^2 \Big).
	\end{equation}
	Consequently,
	\begin{equation} \label{SUPP:eq:cond-2-pf}
		\| \Ubf^\top \Sbf^\top \Sbf (\Xbf_{[n]}^\top)^\perp \|_\F^2 \leq \frac{\varepsilon}{2}\OPT^2
	\end{equation}
	with probability at least $1-\delta/2$. 
	By a union bound, it follows that both \eqref{SUPP:eq:cond-1-pf} and \eqref{SUPP:eq:cond-2-pf} are true with probability at least $1-\delta$. From Lemma~\ref{SUPP:lemma:structural-conds} it therefore follows that \eqref{SUPP:eq:main-statement} is true with probability at least $1-\delta$.
\end{proof}

\section{Detailed Complexity Analysis} \label{SUPP:sec:detailed-complexity}

We provide a detailed complexity analysis in this section to show how we arrived at the numbers in Table~\ref{table:complexity} in the main manuscript.

\subsection{TR-ALS} \label{SUPP:sec:TR-ALS-complexity}

In our calculations below, we refer to steps in Algorithm~\ref{alg:TR-ALS} in the main paper, consequently ignoring any cost associated with e.g.\ normalization and checking termination conditions.

Upfront costs of TR-ALS:
\begin{itemize}
	\item \textbf{Line~\ref{line:tr-als:initialize-cores}: Initializing cores.} 
	This depends on how the initialization of the cores is done. 
	We assume they are randomly drawn, e.g.\ from a Gaussian distribution, resulting in a cost $O(NIR^2)$.
\end{itemize}

Costs per outer loop iteration of TR-ALS:
\begin{itemize}
	\item \textbf{Line~\ref{line:tr-als:compute-subchain}: Compute unfolded subchain tensor.} 
	If the $N-1$ cores are dense and contracted in sequence, the cost is 
	\begin{equation}
		R^{3} (I^2 + I^3 + \cdots + I^{N-1}) \leq R^{3} (N I^{N-2} + I^{N-1}) \leq 2 R^{3} I^{N-1} = O(I^{N-1} R^3),
	\end{equation}
	where we use the assumption $N < I$ in the second inequality.
	Doing this for each of the $N$ cores in the inner loop brings the cost to $O(N I^{N-1} R^3)$.
	
	\item \textbf{Line~\ref{line:tr-als:ls}: Solve least squares problem.} 
	We consider the cost when using the standard QR-based approach described in Section~5.3.3 in the book by \citet{golub2013}.
	The matrix $\Gbf_{[2]}^{\neq n}$ is of size $I^{N-1} \times R^2$. 
	Doing a QR decomposition of this matrix costs $O(I^{N-1} R^4)$. 
	Updating the right hand sides and doing back substitution costs $O(I (I^{N-1} R^2 + R^4)) = O(I^N R^2)$, where we used the assumption that $R^2 < I$.
	The leading order cost for solving the least squares problem is therefore $O(I^N R^2)$.
	Doing this for each of the $N$ cores in the inner loop brings the cost to $O(N I^N R^2)$.
\end{itemize}

It follows that the overall leading order cost of TR-ALS is $NIR^2 + \noiter \cdot N I^N R^2$.

\subsection{rTR-ALS}

In our calculations below, we refer to steps in Algorithm~1 by \citet{yuan2019} with the TR decomposition step in their algorithm done using TR-ALS.

Cost of initial Tucker compression:
\begin{itemize}
	\item \textbf{Line~4: Draw Gaussian matrix.} Drawing these $N$ matrices costs $O(N I^{N-1} K)$.
	\item \textbf{Line~5: Compute random projection.} Computing $N$ projections costs $O(N I^N K)$.
	\item \textbf{Line~6: QR decomposition.} Computing the QR decomposition of an $I \times K$ matrix $N$ times costs $O(N I K^2)$.
	\item \textbf{Line~7: Compute Tucker core.} The cost of compressing each dimension of the input tensor in sequence is 
	\begin{equation}
		I^N K + I^{N-1} K^2 + I^{N-2} K^3 + \cdots + I K^N = O(N I^N K),
	\end{equation}
	where we assume that $K < I$.
\end{itemize}

Cost of TR decompositions:
\begin{itemize}
	\item \textbf{Line~9: Compute TR decomposition of Tucker core.} Using TR-ALS, this costs $O(N K R^2 + \noiter \cdot N K^N R^2)$ as discussed in Section~\ref{SUPP:sec:TR-ALS-complexity}.
	\item \textbf{Line~11: Compute large TR cores.} This costs in total $O(N I K R^2)$ for all cores.
\end{itemize}

Combining these costs and recalling the assumption $R^2 < I$, we get an overall leading order cost for rTR-ALS of $NI^N K + \noiter \cdot N K^N R^2$.

\subsection{TR-SVD}

In our calculations below, we refer to steps in Algorithm~1 by \citet{mickelin2020}.
We consider a modified version of this algorithm which takes a target rank $(R_1, \ldots, R_N)$ as input instead of an accuracy upper bound $\varepsilon$.
For simplicity, we ignore all permutation and reshaping costs, and focus only on the leading order costs which are made up by the SVD calculations.
\begin{itemize}
	\item \textbf{Line~3: Initial SVD.} 
	This is an economy sized SVD of an $I \times I^{N-1}$ matrix, which costs $O(I^{N+1})$; see the table in Figure~8.6.1 of \citet{golub2013} for details.
	
	\item \textbf{Line~10: SVD in for loop.} 
	The size of the matrix being decomposed will be $IR \times I^{N-k} R$. 
	The cost of computing an economy sized SVD of each for $k = 2,\ldots,N-1$ is
	\begin{equation}
		R^3 (I^N + I^{N-1} + \cdots + I^3) \leq 2 R^3 I^N,
	\end{equation}
	where we use the assumption that $N < I$.
\end{itemize}
Adding these costs up, we get a total cost of $O(I^{N+1} + I^N R^3)$.

\subsection{TR-SVD-Rand}

In our calculations below, we refer to steps in Algorithm~7 by \citet{ahmadi-asl2020}.
For simplicity, we ignore all permutation and reshaping costs, and focus only on the leading order costs.
In particular, note that the QR decompositions that come up are of relatively small matrices and therefore relatively cheap.
Moreover, we assume that the oversampling parameter $P$ is small enough to ignore.
\begin{itemize}
	\item \textbf{Line~2: Compute projection.} 
	The matrix $\Cbf$ is of size $I \times I^{N-1}$ and the matrix $\Omegabf$ is of size $I^{N-1} \times R^2$.
	The cost of computing their product is $O(I^N R^2)$.
	
	\item \textbf{Line~6: Computing tensor-times-matrix (TTM) product.} 
	$\Xe$ is an $N$-way tensor of size $I \times \cdots \times I$, and $\Qbf^{(1)}$ is of size $I \times R^2$, so this contraction costs $O(I^N R^2)$.
	
	\item \textbf{Line~11: Compute projections in for loop.} 
	Each of these costs $I^{N-n+1} R^3$.
	The total cost for all iterations for $n = 2, \ldots, N-1$ is therefore
	\begin{equation}
		R^3 (I^{N-1} + I^{N-2} + \cdots + I^2) \leq 2 R^3 I^{N-1},
	\end{equation}
	where we used the assumption that $N < I$.
	
	\item \textbf{Line~15: Compute TTM product in for loop.} 
	Each of these costs $I^{N-n+1} R^3$.
	The total cost for all iterations is therefore 
	\begin{equation}
		R^3 (I^{N-1} + I^{N-2} + \cdots + I^2) \leq 2 R^3 I^{N-1},
	\end{equation}
	where we used the assumption that $N < I$.
\end{itemize}
Adding these costs up, we get a total cost of $O(I^N R^2)$.

\subsection{TR-ALS-Sampled}

In our calculations below, we refer to steps in Algorithm~\ref{alg:TR-ALS-Sampled} in the main paper.

Upfront costs of TR-ALS-Sampled:
\begin{itemize}
	\item \textbf{Line~\ref{line:tr-als-sampled:initialize-cores}: Initializing cores.} 
	This is the same as for TR-ALS, namely $O(NIR^2)$.
	
	\item \textbf{Line~\ref{line:tr-als-sampled:compute-p}: Compute distributions.} 
	For each $n = 2,\ldots, N$, this involves computing the economic SVD of an $I \times R^2$ matrix for a cost $O(I R^4)$, and then computing $\pbf^{(n)}$ from the left singular vectors for a cost of $O(IR)$. 
	This yields a total cost for this line of $O(NIR^4)$.
	
\end{itemize}

We ignore the cost of the sampling in Line~\ref{line:tr-als-sampled:draw-sketch} since it is typically very fast.

Costs per outer loop iteration of TR-ALS-Sampled:
\begin{itemize}	
	\item \textbf{Line~\ref{line:tr-als-sampled:sample-G}: Compute sampled unfolded subchain.} 
	The main cost for this line is computing the product of a sequence of $N-1$ matrices of size $R \times R$, for each of the $J$ sampled slices (see Figure~\ref{fig:core-sampling} in the main paper).
	This costs $O(N R^3 J)$ per inner loop iteration, or $O(N^2 R^3 J)$ per outer loop iteration.
	
	\item \textbf{Line~\ref{line:tr-als-sampled:sample-X}: Sample input tensor.} 
	This step requires copying $O(I J)$ elements from the input tensor. 
	For one iteration of the outer loop, this is $O(N I J)$ elements.
	In practice this step together with the least squares solve are the two most time consuming since it involves sampling from a possibly very large array.
	
	\item \textbf{Line~\ref{line:tr-als-sampled:ls}: Solve least squares problem.}
	The cost is computed in the same way as the least squares solve in TR-ALS, but with the large dimension $I^{N-1}$ replaced by $J$.
	The cost per least squares problem is therefore $O(IJR^2)$, or $O(N I J R^2)$ for one iteration of the outer loop.
	
	\item \textbf{Line~\ref{line:tr-als-sampled:update-p}: Update distributions.}
	For one iteration of the outer loop this costs $O(NIR^4)$.
\end{itemize}

Adding these costs and simplifying, we arrive at a cost 
\begin{equation} \label{SUPP:eq:tr-als-sampled-general-complexity}
	O(N I R^4 + \noiter \cdot N I J R^2).
\end{equation}
If $\varepsilon$ and $\delta$ are small enough, then (\ref{eq:J-bound-2}) simplifies to $J > 4 R^{2N} / (\varepsilon \delta)$. 
Plugging this into \eqref{SUPP:eq:tr-als-sampled-general-complexity} gives us the complexity in Table~\ref{table:complexity}.

\section{Links to Datasets} \label{SUPP:sec:links-to-datasets}

\begin{itemize}
	\item The Pavia University dataset was downloaded from\\ \url{http://lesun.weebly.com/hyperspectral-data-set.html}.
	
	\item The Washington DC Mall dataset was downloaded from\\
	\url{https://engineering.purdue.edu/~biehl/MultiSpec/hyperspectral.html}.
	
	\item The Park Bench video was downloaded from\\
	\url{https://www.pexels.com/video/man-sitting-on-a-bench-853751}.\\
	The video, which is in color, was made into grayscale by averaging the three color bands.
	
	\item The Tabby Cat video was downloaded from\\
	\url{https://www.pexels.com/video/video-of-a-tabby-cat-854982/}.\\
	The video, which is in color, was made into grayscale by averaging the three color bands.
	
	\item The Red Truck images are part of the COIL-100 dataset, which was downloaded from\\
	\url{https://www1.cs.columbia.edu/CAVE/software/softlib/coil-100.php}.
\end{itemize}

\section{Additional Experiment Details} \label{SUPP:sec:additional-experiment-details}

Here we provide additional experiment details, including how the number of ALS iterations and appropriate sketch rates are determined for the experiments in Section~\ref{sec:experiments-decomposition} of the main paper.

\subsection{Randomly Generated Data}

\paragraph{First Experiment}
When determining the number of ALS iterations, TR-ALS is run until the change in relative error is below 1e\textminus6 or for a maximum of 500 iterations, whichever is satisfied first.
The sample size $J$ for TR-ALS-Sampled is started at 200 and incremented by 100.
The embedding dimension $K$ for rTR-ALS is started at $I/10$ and incremented by $I/20$.
Incrementation is done until the error for each method is smaller than 1.2 times the TR-ALS error. 

\paragraph{Second Experiment}
The sketch rate $J$ for TR-ALS-Sampled is now incremented by 1000.  
Incrementation is done until the error for each method is smaller than 1.02 times the TR-ALS error. 
A smaller factor is used compared to the first experiment since the TR-ALS error is much larger.
The other settings remain the same as in the first experiment.

\begin{remark}[Performance of SVD-based methods] \label{SUPP:remark:SVD-performance}
	The SVD-based methods typically require much higher ranks than the ALS-based methods.
	To achieve a similar error to the ALS-based methods in Figures~\ref{fig:experiment-1a} (a) and \ref{fig:experiment-1b} (a) (0.0031 and 0.94, respectively) the \emph{original unaltered implementation} of TR-SVD by \citet{mickelin2020} requires average TR ranks (i.e., $(R_1+R_2+R_3)/3$) in the range of 82--233 and 35--84, respectively.
	TR-SVD-Rand does poorly for the same reason.
\end{remark}

\begin{remark}[Empirical vs.\ theoretical complexity] \label{SUPP:remark:empirical-vs-theoretical-complexity}
	As discussed in Section~\ref{sec:complexity-analysis}, the benefit of our proposed method is that it has a lower complexity than the competing methods.
	In particular, it avoids the $I^N$ factors in the complexity expression.
	It may therefore seem surprising that our method is not the fastest in the experiments.
	For example, in Figure~\ref{fig:experiment-1a} (b) our method is always slower than TR-SVD-Rand, and in Figure~\ref{fig:experiment-1b} (b) it is always slower than both TR-SVD and TR-SVD-Rand.
	The reason for this seeming discrepancy is that we only consider a small range of $I$ values ($I \in [100, 500]$) in those figures and therefore the plots will not necessarily reflect the leading order complexities that are given in Table~\ref{table:complexity}.
	The main cost in TR-SVD-Rand is matrix multiplication which is very efficient, so the hidden constant in the complexity for TR-SVD-Rand is small, and this helps explain why it is faster than our method for the relatively small $I$ values used in Figures~\ref{fig:experiment-1a} and \ref{fig:experiment-1b}.
	Similar comments also apply to the other experiments. 
\end{remark}

\subsection{Highly Oscillatory Functions}

To determine the number of iterations for the ALS-based methods, we run TR-ALS until the change in relative error is less than 1e\textminus3 or for a maximum of 100 iterations, whichever is satisfied first.
The sample size $J$ for TR-ALS-Sampled is started at $2R^2$ and incremented by 100.
The embedding dimension $K$ for rTR-ALS is started at $2$ and incremented by $1$.
If $K \geq I_n$, no compression is applied to the $n$th dimension.
The incrementation is done until the error for each method is smaller than 1.1 times the TR-ALS error.

\subsection{Image and Video Data} \label{SUPP:sec:additional-details-image-and-video}

To determine the number of iterations for the ALS-based methods, we run TR-ALS until the change in relative error is less than 1e\textminus3 or for a maximum of 100 iterations, whichever is satisfied first.
The sample size $J$ for TR-ALS-Sampled is started at $2R^2$ and incremented by 1000.
The embedding dimension $K$ for rTR-ALS is started at $\max_{n \in [N]} I_n / 10$ and incremented by $\max_{n \in [N]} I_n / 20$.
If $K \geq I_n$, no compression is applied to the $n$th dimension.
The incrementation is done until the error for each method is smaller than 1.1 times the TR-ALS error.

In the experiment on the reshaped tensors, each mode (except the mode representing color channels in the Red Truck dataset) of the original tensor is split into two new modes.
Some datasets are also truncated somewhat to allow for this reshaping.
The details are given below.
\begin{itemize}
	\item \textbf{Pavia Uni.}\ is first truncated to size $600 \times 320 \times 100$. 
	This is done by discarding the last elements in each mode via \verb|X = X(1:600, 1:320, 1:100)| in Matlab.
	The tensor is then reshaped into a $24 \times 25 \times 16 \times 20 \times 10 \times 10$ tensor.
	This is done by splitting each original mode into two modes via \verb|X = reshape(X,24,25,16,20,10,10)| in Matlab.
	
	\item \textbf{DC Mall} is truncated to size $1280 \times 306 \times 190$ and then reshaped into a $32 \times 40 \times 18 \times 17 \times 10 \times 19$ tensor similarly to how the Pavia Uni.\ dataset is truncated and reshaped.
	
	\item \textbf{Park Bench} does not require any truncation. 
	The original tensor is reshaped into a $24 \times 45 \times 32 \times 60 \times 28 \times 13$ tensor similarly to how the Pavia Uni.\ dataset is reshaped.
	
	\item \textbf{Tabby Cat} does not require any truncation. 
	The original tensor is reshaped into a $16 \times 45 \times 32 \times 40 \times 13 \times 22$ tensor similarly to how the Pavia Uni.\ dataset is reshaped.
	
	\item \textbf{Red Truck} does not require any truncation.
	The original tensor is reshaped into a $8 \times 16 \times 8 \times 16 \times 3 \times 8 \times 9$ tensor.
	Here, all modes have been split into two modes, except for the mode corresponding to the three color channels which is left as it is.
	This is done via \verb|X = reshape(X,8,16,8,16,3,8,9)| in Matlab.
	
\end{itemize}
Detailed results for the experiments on the reshaped tensors are shown in Table~\ref{SUPP:tab:reshaped-results}.

\sisetup{round-mode=places, table-number-alignment=right, table-text-alignment=right}
\begin{table}[ht!]
	\centering
	\caption{
		Decomposition results for reshaped real datasets with target rank $R=10$.
		The SVD-based methods cannot handle any of these reshaped datasets since they require $R_0 R_1 \leq I_1$.
		The ALS-based methods all fail on the reshaped Park Bench dataset due to Matlab running out of memory.
		Time is in seconds. 
	} 
	\label{SUPP:tab:reshaped-results}
	\begin{tabular}{
			l
			S[round-precision=2, table-figures-decimal=2, table-figures-integer=1]
			S[round-precision=1, table-figures-decimal=1, table-figures-integer=4]
			S[round-precision=2, table-figures-decimal=2, table-figures-integer=1]
			S[round-precision=1, table-figures-decimal=1, table-figures-integer=4]
			S[round-precision=2, table-figures-decimal=2, table-figures-integer=1]
			S[round-precision=1, table-figures-decimal=1, table-figures-integer=4]
			S[round-precision=2, table-figures-decimal=2, table-figures-integer=1]
			S[round-precision=1, table-figures-decimal=1, table-figures-integer=4]
			S[round-precision=2, table-figures-decimal=2, table-figures-integer=1]
			S[round-precision=1, table-figures-decimal=1, table-figures-integer=3]
		}
		\toprule
		& \multicolumn{2}{c}{Pavia Uni.} & \multicolumn{2}{c}{DC Mall} & \multicolumn{2}{c}{Park Bench} & \multicolumn{2}{c}{Tabby Cat} & \multicolumn{2}{c}{Red Truck}\\
		\cmidrule(lr){2-3}
		\cmidrule(lr){4-5}
		\cmidrule(lr){6-7}
		\cmidrule(lr){8-9}
		\cmidrule(l){10-11}
		Method 				 & {Error} & {Time} & {Error} & {Time}  & {Error} & {Time}  & {Error} & {Time}  & {Error} & {Time}  \\
		\midrule
		TR-ALS 				 & 0.28 & 1372.2 & 0.29 & 3947.7 & \xmark & \xmark & 0.15 & 6629.0 & 0.25 & 546.3 \\
		rTR-ALS 			 & 0.31 &  944.7 & 0.31 & 2575.0 & \xmark & \xmark & 0.17 & 3416.4 & 0.26 & 423.5 \\
		TR-ALS-S. (proposal) & 0.31 &    3.4 & 0.31 &    5.8 & \xmark & \xmark & 0.17 &    2.3 & 0.27 &   5.8 \\
		\midrule
		TR-SVD 				 & \xmark & \xmark & \xmark & \xmark & \xmark & \xmark & \xmark & \xmark & \xmark & \xmark \\
		TR-SVD-Rand 		 & \xmark & \xmark & \xmark & \xmark & \xmark & \xmark & \xmark & \xmark & \xmark & \xmark \\
		\bottomrule
	\end{tabular}
\end{table}

\subsection{Rapid Feature Extraction for Classification} \label{SUPP:sec:additional-details-feature-extraction}

The ALS-based algorithms all run until the change in relative error is below 1e\textminus4.
The embedding dimension $K$ for rTR-ALS is 20 for the first and second modes, and 200 for the fourth mode. 
No compression is applied to the third mode since it is already so small.
TR-ALS-Sampled uses the sketch rate $J = 1000$.

\end{document}